\documentclass[12pt,leqno]{amsart}
\usepackage[dvips]{graphics}
\usepackage{amssymb}
\usepackage{color}

\numberwithin{equation}{section}
\newdimen\vintkern\vintkern12pt
\def\vint{-\kern-\vintkern\int}

\newtheorem{thm}{Theorem}[section]

    \newtheorem{lem}[thm]{Lemma}
\newtheorem{cor}[thm]{Corollary}
\newtheorem{prop}[thm]{Proposition}

 \theoremstyle{definition}

\theoremstyle{remark}
\newtheorem{rem}{Remark}[section]

\newcommand{\tref}[1]{Theorem~\ref{#1}}
\newcommand{\cref}[1]{Corollary~\ref{#1}}
\newcommand{\pref}[1]{Proposition~\ref{#1}}

\newcommand{\lref}[1]{Lemma~\ref{#1}}

\newcommand{\R}{\mathbb{R}}
\newcommand{\N}{\mathbb{N}}

\newcommand{\Area}{\operatorname{Area}}
\newcommand{\ap}{\operatorname{ap}}
\newcommand{\Id}{\mathrm{Id}}
\newcommand{\md}{\operatorname{md}}
\newcommand{\apmd}{\ap\md}

\newcommand{\trace}{\operatorname{tr}}

\begin{document}
\pagebreak


\title{Isoperimetric  characterization of upper   curvature bounds}

\author{Alexander Lytchak}

\address
  {Mathematisches Institut\\ Universit\"at K\"oln\\ Weyertal 86 -- 90\\ 50931 K\"oln, Germany}
\email{alytchak@math.uni-koeln.de}

\author{Stefan Wenger}

\address
  {Department of Mathematics\\ University of Fribourg\\ Chemin du Mus\'ee 23\\ 1700 Fribourg, Switzerland}
\email{stefan.wenger@unifr.ch}

\date{\today}


\begin{abstract}
 We  prove that a proper geodesic metric space has non-positive curvature in the sense of Alexandrov if and only if  it satisfies the Euclidean isoperimetric inequality for
 curves. Our result  extends to non-geodesic spaces and non-zero curvature bounds.
 \end{abstract}

\maketitle

\renewcommand{\theequation}{\arabic{section}.\arabic{equation}}
\pagenumbering{arabic}

\section{Introduction}
\subsection{Main result}
 We say that a metric space $X$ satisfies the Euclidean isoperimetric inequality for curves if any closed Lipschitz curve $\gamma:S^1 \to X$  bounds  a Lipschitz map of  the unit  disc $v: \bar D\to X$
whose parametrized Hausdorff area  is at most $\frac 1 {4\pi} \ell ^2 _X(\gamma ) $. Here, $\ell _X (\gamma )$ denotes the length of $\gamma$ in $X$.
We refer to the first two sections below for   the notion of parametrized Hausdorff area and other  basic notions of metric geometry involved in the   following main theorem of the present paper.

\begin{thm}  \label{thmfirst}
Let $X$ be a proper metric space in which any pair of points is connected by a curve  of finite length. Let $X^i$ denote the set $X$ with the induced length metric. The space $X^i$ is  ${\rm CAT}(0)$ if and only if  $X$ satisfies the Euclidean isoperimetric inequality for curves.
\end{thm}

This result provides an analytic access to upper curvature bounds and can be used to recognize upper curvature bounds without being able to identify geodesics or angles.  Such situations often appear in metric constructions, cf. \cite{Alexander}.
\tref{thmfirst} admits a natural generalization to non-zero curvature bounds, see  \tref{thmmain} below.

The ``only if part" of our theorem is folklore and follows as    an easy consequence of Reshetnyak's majorization theorem,
\cite{Resh-majorization}.
 With a  different definition of area the ``only if part" already appears in \cite{Aleksandrov} at the very origin of the theory of spaces with upper curvature bounds.

 Results closest to the much subtler ``if part" of our theorem have been proven
in \cite{BR}  and  \cite{Resh-isop} in the case of surfaces. E. Beckenbach and T. Rado proved in \cite{BR} our \tref{thmfirst} for smooth $2$-dimensional Riemannian manifolds, finding a connection between log-subharmonicity, isoperimetric inequalities and curvature bounds.
  In \cite{Resh-isop} the result of
\cite{BR} was extended  to some singular    surfaces  and non-zero curvature bounds.

\subsection{Main  idea} Essentially, the strategy of   our proof  of  the "if part" is to reduce the problem  in a general metric space  to the situation considered in \cite{Resh-isop}.
Recall  the following  simple consequence of the Gauss equation in Riemannian geometry:  a minimal surface has curvature no larger  than the ambient space. We reverse this idea and find a curvature bound for the total space by proving that all minimal discs have the corresponding curvature bound:

\begin{thm} \label{thmminimal}
Let $X$ be a  proper  metric space  which satisfies the Euclidean isoperimetric inequality for curves.
 Let $\Gamma$ be a Jordan curve of finite length  in $X$ and let $u:\bar D\to X$ be a solution of the Plateau problem in the space $X$ for the boundary curve $\Gamma$.
Then the  intrinsic minimal disc $Z$ associated with $u$ is a ${\rm CAT}(0)$ space.
\end{thm}

We refer to \cite{LW}, \cite{LW-intrinsic} and Section \ref{secplateau} below for
the notion of a solution of the Plateau problem  and the associated  intrinsic  minimal disc.
By definition of the intrinsic  minimal disc $Z$, the solution of the Plateau problem  $u$ in \tref{thmminimal} factorizes as $u=\bar u\circ P$ for a surjective map $P:\bar D\to Z$ and a $1$-Lipschitz map $\bar u:Z\to X$.
Moreover, $\bar u$ sends
the boundary circle $\partial Z$ of $Z$ in an arclength preserving way onto $\Gamma$.

 It is not difficult  to see that  \tref{thmminimal} implies \tref{thmfirst}. Assume for simplicity that the proper space $X$
  with the Euclidean isoperimetirc inequality for curves is a length space. The existence of a solution $u$ of the Plateau problem
for any rectifiable Jordan curve $\Gamma$  in $X$ is  proved in \cite{LW}, generalizing \cite{Mor48} to the setting of metric spaces.
In order to prove that $X$ is ${\rm CAT} (0)$, one needs to prove that any Jordan triangle in $X$ is thin, cf.  Section \ref{sec:curvature} below.
However, \tref{thmminimal} implies that $\Gamma$ is thin in the intrinsic minimal   disc  $Z$.  Using the map $\bar u$
(which majorizes $\Gamma$ in the sense of \cite{Resh-majorization} and \cite[Section 9.8]{AKP})  this easily implies that $\Gamma$ is thin in $X$ as well.

\subsection{Main steps}
The proof of \tref{thmminimal} involves several steps. First, a special case of the Blaschke-Santalo inequality implies that, among normed planes, only the Euclidean plane satisfies the Euclidean isoperimetric inequality, \cite{Tho96}. The quasi-convexity of the Hausdorff area proved in \cite{BI} and a natural blow-up argument imply that $X$ has only ``Euclidean tangent spaces'',  at least as far as infinitesimal properties of Sobolev maps with values in $X$ are concerned.  This is the property (ET) introduced in \cite{LW}, which greatly simplifies the description of Sobolev maps and solutions of  the Plateau problem.

In particular, the solution of the Plateau problem  $u$ in \tref{thmminimal} is a conformal map, see Section \ref{secSob}. Thus, there exists a non-negative Borel function $f\in L^2(D)$, \emph{the conformal factor of $u$}, such that for
 \emph{almost all curves} $\gamma$ in $D$ the length of the image of $\gamma$ under $u$ is controlled by $f$:
 \begin{equation} \label{eq:conform}
 \ell _X(u\circ \gamma ) =\int  _{\gamma} f.
 \end{equation}
 The next step goes back to  \cite{BR} and shows that the isoperimetric inequality forces $f$ to be log-subharmonic.

The subsequent  step,  contained in  \cite{Resh-isop},  relates log-subharmonicity of conformal factors to non-positive curvature in the sense of Alexandrov. More precisely, the length metric defined on $D$ by setting   the length of \emph{every} rectifiable curve $\gamma \subset D$ to be  $\int _{\gamma} f$  is locally ${\rm CAT}(0)$. The  metric space $Y$ defined in this way is  intimately  related to the intrinsic minimal disc $Z$.  The only difference is that \eqref{eq:conform} holds in $Y$ for  \emph{all} and in $Z$ for \emph{almost all}
rectifiable curves $\gamma$.  In particular, we have a $1$-Lipschitz map $I:Y\to Z$ which preserves the length of almost all curves.

 The final, rather subtle step  is  devoted to the proof that the spaces  $Z$ and (the completion of) $Y$   are identical.  While the analytically defined conformal factor $f$ controls the lengths of \emph{almost all} curves in $Z$, it cannot control the lengths of all curves: contracting  one interval in $D$  to a point does not change the conformal factor. In particular, $f$ does not a priori control the most important boundary curve. Applying some cutting and pasting tricks we reduce the final step  to the question whether the length of the boundary curve ``is controlled by the conformal factor''.
 Using  general structural results about the intrinsic minimal discs   obtained in \cite{LW-intrinsic},  the final step reduces to the following:

\begin{thm} \label{thm4}
Let $Z$ be a geodesic metric space homeomorphic to the closed disc $\bar D$. Denote by $\partial Z$ the boundary circle and     assume that  $Z\setminus \partial Z$ is locally $\rm{CAT} (0)$.  Then the following are equivalent.
  \begin{enumerate}
  \item $Z$ is $\rm{CAT}(0)$.
  \item $Z\setminus \partial Z$ with the metric induced from $Z$ is a length space.
  \item For
 any Jordan curve $\eta \subset Z$
the open disc $J_{\eta}$ enclosed by $\eta$ in $Z$  satisfies $\mathcal H^2 (J_{\eta} ) \leq \frac 1 {4\pi} \cdot \ell ^2 _Z(\eta) $.
\end{enumerate}
\end{thm}
%

 Throughout the text, we denote by $\mathcal H^2$ the $2$-dimensional Hausdorff measure.  Thus,  condition (3) in \tref{thm4} is the geometric (unparametrized)  version of the Euclidean isoperimetric inequality.
This theorem does not sound very surprising. The implication from (1) to (3) is an easy  consequence of Reshetnyak's  majorization theorem. The equivalence of (2) and (1) is not very  difficult either.
In contrast, the proof of the main implication from (3) to (1)   is rather long and technical and comprises  one half  of this paper.    One might have the following example in mind in order to  grasp the problem one faces when trying to prove this implication. Start with a complicated Jordan curve $\Gamma$ in $\R^2$, for example Koch's snowflake.  Define $Z$ as the closure of the  Jordan domain  of $\Gamma$ but let the boundary curve $\Gamma\subset Z$ have some finite length without changing the lengths outside of $\Gamma$.  (To make the picture more complicated, change the metric inside the Jordan domain by a smooth conformal factor in such a way that the curvature is everywhere non-positive and tends to $-\infty$ in the neighborhood of $\Gamma$).  The arising space $Z$ is not ${\rm CAT}(0)$. Therefore,
the proof of \tref{thm4} must detect in this  space $Z$ Jordan curves which violate the Euclidean  isoperimetric inequality. Since $Z\setminus \partial Z$ is locally $\rm{CAT} (0)$, parts of these curves must be contained in the boundary $\Gamma$,   where the geometry is particularly complicated.

\subsection{Generalization to non-zero curvature bounds} \tref{thmfirst} generalizes to other curvature bounds. The extension  is achieved along the same route and involves only  minor difficulties of notational and technical nature. In order to formulate the statement we introduce the notion of a \emph{Dehn function}. Let $X$ be a metric space.
 Let $\delta:(0,\infty) \to [0,\infty]$ be a function. We say that $X$ satisfies the $\delta$-isoperimetric inequality,
 if for any $r>0$, any  Lipschitz curve $\gamma :S^1\to X$
of length $\leq r$ bounds a Lipschitz disc $u: \bar D\to X$ of parametrized Hausdorff area $\leq \delta (r)$.
The Dehn function $\delta _X$  of $X$ (with respect to Lipschitz discs)  is the infimum of all functions
$\delta :(0,\infty) \to [0,\infty]$ for which $X$ satisfies the $\delta $-isoperimetric inequality.

  For any real number $\kappa $  we consider the simply connected
space form $M_{\kappa } ^2$ of curvature $\kappa$ and denote by $R_{\kappa} \in (0, \infty]$ twice the diameter of $M_{\kappa} ^2$.  Let $\delta_{\kappa}$
 be  the Dehn function  of $M_{\kappa} ^2$.
 Now  we can state the generalization of \tref{thmfirst}   to non-zero curvature bounds.

\begin{thm}  \label{thmmain}
Let $X$ be a proper metric space in which any pair of points is connected by a curve  of finite length. Let $X^i$ be  the set $X$ with the induced length metric. The space
$X^i$ is  ${\rm CAT}(\kappa )$  if and only if the Dehn function $\delta _X$ of $X$ satisfies $\delta _X \leq \delta _{\kappa } $
on the interval $(0,R_{\kappa})$.
\end{thm}

\subsection{Structure of paper and final comments}
The paper consists of two parts and one appendix. The first part, which relies heavily on the existence and regularity theory of solutions of the  Plateau problem, reduces \tref{thmfirst} and \tref{thmminimal}  to  \tref{thm4}. We closely follow the plan sketched above. The second part is devoted to the proof of \tref{thm4}.  It consists of purely $2$-dimensional metric geometry.
  The structure of this part is explained in Section \ref{Intro2}.   In the appendix we explain the  minor additional difficulties arising in the proof of \tref{thmmain}  and sketch
 the solutions of these problems.

\begin{rem}
Once \tref{thmfirst} has been proven, the statements of \tref{thmminimal} and \tref{thm4} can be strengthened, see \cite{Petrunin-old}, \cite{Petrunin} and \cite{LW-parameter}.
\end{rem}

\begin{rem}
It would be  interesting to extend  \tref{thmfirst}  to the non-proper and to the coarse settings, similarly to  the results in \cite{Wengerhyperbolic}.
\end{rem}

\begin{rem} \label{rmk-ht}
In convex and metric geometry  there are many natural ways to measure area of $2$-rectifiable sets  and Lipschitz discs besides  the Hausdorff area.
The most famous among such definitions of area are  Gromov's mass$^*$ and the Holmes-Thompson definition of area. We refer to \cite{Tho96} and \cite{LW-intrinsic} for lengthy discussions on definitions of area.
As explained above, the first step of the proof of our main theorem   uses an inequality from convex geometry to exclude all
non-Euclidean tangent planes. The argument applies to all quasi-convex definitions of area $\mu$ with the following property: among all normed planes only the  Euclidean plane satisfies the Euclidean isoperimetric inequality with respect to $\mu$.  For the Holmes-Thompson definition of area $\mu ^{ht}$
the Euclidean isoperimetric inequality holds sharply for all normed planes. Thus, \tref{thmfirst} is valid for any quasi-convex definition of area $\mu$ which satisfies $\mu \geq \mu^{ht}$ on all normed planes with equality only on the Euclidean plane.  In particular, \tref{thmfirst} remains  valid for Gromov's mass$^{\ast}$  and for Ivanov's  "inscribed Riemannian"   definition of area.

The validity of the Euclidean isoperimetric inequality for curves with respect to the Holmes-Thompson definition of area  might  be related to other forms of
 convexity beyond ${\rm CAT}(0)$.
\end{rem}

\begin{rem}
If the constant $\frac 1 {4\pi}$ in the formulation of the Euclidean isoperimetric inequality is replaced by any smaller
constant, then the space $X^i$  in \tref{thmfirst} turns out to be a tree,  \cite{LW-asymptotic}, \cite{Wengerhyperbolic}.
\end{rem}

\begin{rem}
 As a consequence of  \tref{thmfirst},  a  proper geodesic metric  space with Euclidean isoperimetric inequality for curves must be contractible.
It would be interesting to know whether any topological conclusions can be drawn, if the constant $\frac 1 {4\pi}$ is replaced by a slightly larger constant
$\frac 1 {4\pi} <C <\frac 1 {2\pi}$.
\end{rem}

\subsection{Acknowledgements} The authors would like to thank  Werner Ballmann, Anton Petrunin and Stephan Stadler for helpful conversations.
 Research relevant to the present paper has been carried out since 2009 at the following institutions: University of Illinois at Chicago, MSRI at Berkeley, University of M\"unster, University of Cologne,  University of Fribourg. The authors wish to thank these institutions. The first author was supported in part by SFB 878 {\it Groups, Geometry and Analysis} and by a Heisenberg grant from the DFG. The second author was partially supported by NSF Grant DMS-0956374, NSF CAREER Grant DMS-1056263, and Swiss National Science Foundation SNF Grants 153599 and 165848.

\vspace{2cm}

\centerline{\bf{Part I. Structure of minimal discs.}}

\vspace{0.5cm}

\section{Basics on metric spaces} \label{secpre2}
\subsection{Notation}
The Euclidean norm of a vector $v\in\R^2$ is denoted by $|v|$.  We denote the open unit disc in $\R^2$ by $D$.  Connected open subsets of $\R^2$ will be called domains.
A metric space is called proper if its closed bounded subsets are compact. We will denote distances in a metric space $X$ by $d$ or $d_X$.
Let $X=(X,d)$ be a metric space. The open ball  in $X$ of radius $r$ and center $x_0\in X$ is denoted by
$$B(x_0,r)=B_X(x_0,r) = \{x\in X: d(x_0, x)<r\}. $$
%

A Jordan curve in $X$ is a subset $\Gamma\subset X$ which is homeomorphic to $S^1$. Given a Jordan curve $\Gamma\subset X$, a continuous map $\gamma \colon S^1\to X$ is called a weakly monotone parametrization of $\Gamma$ if $\gamma $ is a uniform limit of homeomorphisms $\gamma _i\colon S^1\to\Gamma$.
For $m\geq 0$, the $m$-dimensional Hausdorff measure on $X$ is denoted by $\mathcal H^m=\mathcal H^m_X$. The normalizing constant is chosen in such a way that on Euclidean $\R^m$ the Hausdorff measure $\mathcal H^m$ equals the Lebesgue measure $\mathcal L^m$.

 If no confusion is possible we will identify parametrized curves and their unparametrized images and denote them by the same symbol.
The length of a curve $\gamma $ in a metric space $X$ will be denoted by $\ell_X(\gamma )$ or simply by $\ell(\gamma)$.
 A continuous curve of finite length is called \emph{rectifiable}.
A (local) \emph{geodesic} in a space $X$ is a (locally) isometric map  from an interval to $X$.
   A space $X$ is called \emph{a geodesic space} if any pair of points in $X$ is connected by a geodesic.
A space $X$ is a \emph{length space}
if for all $x,y\in X$ the distance $d(x,y)$ equals $\inf \{\ell_X (\gamma ) \}$, where
$\gamma$ runs over the set of all curves connecting $x$ and $y$.


\subsection{Length metric associated with a map} \label{associated}
We refer the reader to \cite{BBI01}, \cite{Pet-intrinsic},  \cite{LW-intrinsic} for  discussions of the following construction and related topics.
Let $X',X$ be metric spaces. Let $u:X'\to X$ be a continuous map. Assume that for any $y_1,y_2 \in X'$ there exists a continuous curve
$\gamma :I\to X'$ connecting $y_1$ and $y_2$ such that the curve $u\circ \gamma$ has finite length.  Then we let $d_u (y_1,y_2) \in [0,\infty)$
 be the infimum of lengths of all such curves $u\circ \gamma$. The so defined function $d_u:X'\times X'\to [0,\infty )$ is a pseudo-distance
on the set $X'$. The corresponding metric space $Z_u$, which arises from $X'$ by identifying pairs of points with $d_u$-distance $0$, is a length
space.  We will call it the length \emph{metric space associated with the map $u$}.

By construction, the space $Z_u$ associated with  the map $u$ comes with a canonical, possibly non-continuous, surjective projection $P:X'\to Z_u$ and
a $1$-Lipschitz map $\bar u:Z_u\to X$ such that $u=\bar u\circ P$.

The most prominent example of this construction is given as follows. Let $X$ be a metric space in which any pair of points is connected by a curve of finite length.
Then the \emph{length space $X^i$ associated to $X$} is the special case $X^i=Z_u$ of the above construction for the identity map $u=\Id:X\to X$.
If $X$ is proper then, due  to the theorem of Arzela-Ascoli, any pair of points in $X$ is connected by a curve of shortest length. Therefore  the space $X^i$ is a geodesic space.
The completeness of $X$ implies that $X^i$ is complete as well. The $1$-Lipschitz map $\bar u :X^i\to X$ from above  is the identity in this case. The map $P =\bar u ^{-1}:X\to X^i$  need not
 be continuous, but it sends curves of finite length in $X$ to continuous curves of the same length in $X^i$.

\subsection{Polygons and triangles}  A polygon in a metric space $X$ is a closed curve $\gamma:[a,b] \to X$ such that
for some $a=t_1\leq ....\leq t_n =b$ all  restrictions $\gamma :[t_i,t_{i+1}] \to X$   are geodesics.  A Jordan curve $\Gamma$  which
 can be parametrized as a polygon  will be called a
 \emph{Jordan polygon}.   If $n= 3$
we obtain the notions of a \emph{triangle} and  \emph{Jordan triangle}, respectively.

\subsection{Parametrized area of Lipschitz maps} \label{subsecarealip}
 Let $K$ be a Borel subset of $\R^2$  and let $u:K\to X$
 be a Lipschitz map into some metric space $X$.  The set $K$ can be  decomposed as a disjoint union $K=\cup _{i=1} ^{\infty}  K_i \cup A$ in such a way that
 all $K_i$ are compact and $A$ has measure $0$, and such that $u:K_i \to X$ is
either injective or $\mathcal H^2 (u(K_i))=0$, see \cite{Kir94}.   The parametrized area of $u$,  which generalizes the classical parametrized area of smooth maps is given by (see  \cite[Subsection 2.4]{LW-intrinsic}):
   $$\Area(u):= \sum _{i=1} ^{\infty}  \mathcal H^2
(u(K_i))= \int _{u(K)} N (x)  \; d \mathcal  H^2  (x) \, ,$$
 where $N(x)$ is the cardinality of $u^{-1} (x)$.  Alternatively, the parametrized area can be computed by a metric transformation  formula  cf. \cite[p.3]{LW}.

For any biLipschitz homeomorphism $F:K_0\to K \subset \R^2$ the para\-metrized areas of $u:K\to X$ and $u_0=u\circ F:K_0\to X$ coincide.

\section{Upper curvature bounds} \label{sec:curvature}

\subsection{Definition}
For a triangle $\Gamma$ in a metric space $X$ we consider the (unique up to Euclidean motions)  comparison triangle  $\Gamma _0 \subset \R^2$  with the same side-lengths as $\Gamma$. The triangle $\Gamma$ is  called \emph{thick}  (more precisely $0$-thick) if there are points on $\Gamma_0$
 which have smaller distance than the corresponding points on $\Gamma $, cf. \cite{ballmann}.  Otherwise the triangle is called \emph{thin} (or \emph{${\rm CAT}(0 )$-triangle} in the terminology  of \cite{ballmann}).

A complete geodesic  metric space $X$ is ${\rm CAT}(0)$  if  there are no thick triangles in
$X$.  The following observation allows the restriction to Jordan triangles:

\begin{lem} \label{Jordanexist}
Let $X$ be a complete geodesic metric space.
  If $X$ is not ${\rm CAT}(0 )$ then there exists a thick Jordan triangle in $X$.
\end{lem}

\begin{proof}
If there are two different geodesics between a pair of points, then we find parts of these geodesics that build a Jordan curve.  This Jordan curve is a  geodesic bigon, a degenerate case of a triangle, which is automatically thick.

Otherwise geodesics  are uniquely determined by their endpoints.  Given a thick triangle with vertices  $A_1,A_2,A_3$, we find  a uniquely determined Jordan triangle with vertices  $A_1',A_2 ',A_3'$  in the union of the sides, by taking $A_i '$ to be the last common point of the sides $A_iA_j$ and $A_i A_k$.  If the triangle $A_1'A_2'A_3'$ is thin, then so is the triangle $A_1A_2A_3$ by Alexandrov's lemma,
cf. \cite[Lemma 3.5]{ballmann}. Thus we have found  a thick Jordan triangle $A_1'A_2'A_3'$ in $X$.
\end{proof}

\subsection{Majorization theorem}
Let $X$ be a ${\rm CAT}(0 )$ space. Due to the majorization  theorem of Reshetnyak, \cite{Resh-majorization},   any closed curve $\gamma:[0,l]\to X$ parametrized by arclength is \emph{majorized} by a  closed convex set
$\bar \Omega \subset \R^2$ in the following sense, cf. \cite{AKP} and \cite{ballmann}. There exists a simple closed parametrization by arclength  $\eta :[0,l]\to \R^2$ of the boundary  $\partial \Omega$ and
a  $1$-Lipschitz map $M:\bar \Omega \to X$ such that $M\circ \eta =\gamma$.
Then, for any biLipschitz parametrization $F:\bar D\to \bar \Omega$, the area of the  Lipschitz disc $M\circ F$ is bounded by
 $\Area (M\circ F)\leq \mathcal H^2 (\Omega)$.  The  isoperimetric inequality in $\R^2$  yields
   $ \Area (M\circ F) \leq \frac 1 {4\pi} l ^2$.
 Now it is easy to deduce:

 \begin{lem} \label{lem:onlyif}
 Let $X$ be a ${\rm CAT}(0)$ space. Then any Lipschitz curve $\gamma :S^1\to X$ of length $l$  is the boundary of a Lipschitz map
 $u:\bar D\to X$ with $\Area (u) \leq \frac 1 {4\pi} l ^2$.
 \end{lem}

\begin{proof}
Let $\gamma _0:S^1\to X$ be a parametrization of $\gamma$ proportional to arclength.  The existence of a Lipschitz map $u_0:\bar D\to X$ extending $\gamma_0$ with the right bound on the area follows from the paragraph preceding the lemma. We attach to $u_0$ a Lipschitz annulus of zero area connecting $\gamma_0$ and $\gamma$ by a linear reparametrization,  cf. \cite[Lemma 3.6]{LW-asymptotic}.
The arising Lipschitz disc $u$ has the same area as $u_0$ and provides the required filling of $\gamma$.
\end{proof}

\subsection{Curvature bounds via majorization}
The majorization theorem is  closely related to the following  observation.

\begin{lem} \label{curvmaj}
Let $X$ be a complete geodesic metric space. The space $X$ is ${\rm CAT}(0 )$ if and only if for any Jordan triangle $\Gamma \subset X$  there exists a ${\rm CAT}(0)$ space $Z$ and a $1$-Lipschitz map $F:Z\to X$ which sends some closed rectifiable  curve $\Gamma ' \subset Z $ in
an  arclength preserving way onto $\Gamma$.
\end{lem}

\begin{proof}
If $X$ is ${\rm CAT}(0)$ then, for any triangle $\Gamma \subset X$,  we can  take $Z=X$ and $F=\Id :Z\to X$.
Now  assume that any Jordan triangle  in $X$ is majorized by  a ${\rm CAT}(0)$ space $Z$ as
in the formulation of the lemma.  In order to prove that $X$ is ${\rm CAT}(0)$, we only need to prove that any Jordan triangle $\Gamma$ is thin.  Consider a majorization $F:Z\to X$ of the triangle $\Gamma$.
Then the preimage in $\Gamma '$ of any geodesic contained in  $\Gamma$ is a geodesic in $Z$ of the same length, cf. \cite[p.88]{AKP}.  Hence $\Gamma  '$ is a geodesic triangle in $Z$ with the same side-lengths as $\Gamma$, thus  $\Gamma$ and $\Gamma '$ have the same comparison triangle $\Gamma _0 $ in $\R^2$.
Since $Z$ is ${\rm CAT}(0)$, the triangle  $\Gamma '$ is thin. Since $F:\Gamma ' \to \Gamma $ is $1$-Lipschitz
we deduce that $\Gamma$ is thin as well.
\end{proof}

\subsection{Local curvature bounds}
A metric space $X$  has non-positive curvature  if any point in $X$  has a ${\rm CAT}(0 )$ neighborhood.
  A complete geodesic metric space  $X$ of non-positive curvature is ${\rm CAT} (0)$ if and only if $X$ is simply connected,
by a version of the theorem of Cartan-Hadamard \cite[Section 6]{ballmann}.

\subsection{Reshetnyak's gluing theorem}\label{subsec:gluing}
Let $X^{\pm}$ be ${\rm CAT}(0 )$ spaces with closed convex subsets  $A^{\pm} \subset X^{\pm}$. If $G:A^+\to A^-$ is an isometry then the space $X$
arising from gluing   $X^+$ and $X^-$ along the isometry $G$ is ${\rm CAT}(0 )$, cf. \cite[Theorem 9.1.21]{BBI01}. Localizing the statement we see that a gluing of two spaces of non-positive curvature  along isometric locally convex subsets is again a space of non-positive curvature.

\section{Generalities on Sobolev maps} \label{secSob}
 We assume some knowledge of
Sobolev maps with values in  a metric space and refer to \cite{HKST15}, \cite{Res97}, \cite{KS93} \cite{LW}, \cite{LW-intrinsic} and references
therein for explanations.    Let $\Omega$ be a bounded domain in $\R^2$ and $X$ be a complete metric space.
 A map $u\in L^2 (\Omega,X)$ is contained in the (Newton-) Sobolev space $N^{1,2} (\Omega,X)$
if there exists a Borel  function $\rho \in L^2 (\Omega )$
such that for \emph{$2$-almost all Lipschitz curves} $\gamma: [a,b]=I \to \Omega$  the composition $u\circ \gamma$ is  continuous  and
\begin{equation} \label{eq-n1p}
 \ell_X(u\circ \gamma ) \leq \int _{\gamma} \rho :=\int _a ^b \rho (\gamma (t)) \cdot |\gamma '(t)| \, dt \,.
\end{equation}
We refer to \cite{HKST15} for a thorough discussion of the notion of $2$-almost all curves.  For the present paper it is sufficient to know that
for any biLipschitz embedding  $F:I\times I\to \Omega$ and almost all $t\in I$  inequality  \eqref{eq-n1p} holds true for the curve $\gamma _t(s)=F(t,s)$.
 There exists a  minimal function  $\rho=\rho _u$ satisfying the condition above, uniquely defined up to sets of measure $0$. It will be called  the \emph{generalized  gradient} of $u$.  The integral $\int _{\Omega} \rho ^2 _u (z) dz$ coincides  with
   the  Reshetnyak energy, see \cite{Res97}, \cite{LW}, which we denote by $E_+ ^2 (u)$.

 Let $u\in N^{1,2} (\Omega,X)$ be arbitrary. For almost all $z\in \Omega$ there exists
a seminorm
$\apmd u_z$ on $\R^2$ called the approximate metric differential, such that the following conditions hold true, \cite{Kar07}, \cite[Section 4]{LW}, \cite[Lemma 3.1]{LW-intrinsic}. The map $z\mapsto \apmd u_z$ into the space of seminorms has a Borel measurable representative.
For $2$-almost all curves $\gamma:I\to \Omega$ we have:
 \begin{equation} \label{eq:almostall}
\ell_X(u\circ \gamma )=\int _I \apmd u_{\gamma (t)}  (\gamma '(t)) dt.
\end{equation}
Moreover, for almost any $z\in \Omega$ we have $\rho _u (z) = \sup _{v\in S^1} \apmd u_z (v)$.

There is a countable, disjoint decomposition $\Omega =S\cup _{1\leq i<\infty} K_i$ into a set $S$ of measure zero   and compact subsets $K_i$ such that the restriction of $u$ to any $K_i$ is Lipschitz continuous. The (parametrized Hausdorff) area of the Sobolev map $u$ is defined to be $\Area (u):= \sum _{i=1} ^{\infty} \Area (u_i)$, where $u_i$ denotes the Lipschitz continuous restriction of $u$ to $K_i$.
This number $\Area (u)$ is finite, independent of the decomposition and generalizes the area of Lipschitz discs,
cf. \cite[Subsection 3.6]{LW-intrinsic}.

A map $u\in N^{1,2} (\Omega,X)$ is called conformal if at almost all $z\in \Omega$ the seminorm    $\apmd u_z$ is  a multiple $f(z)\cdot s_0$ of the standard Euclidean norm $s_0$ on $\R^2$. In  this case,
$f\in L^2 (\Omega)$ will be called the conformal factor of $u$. The conformal factor $f$ of a conformal map $u\in N^{1,2} (\Omega,X)$ coincides with the generalized gradient  $\rho _u$. In the conformal case, equation
\eqref{eq:almostall} therefore simplifies to
 \begin{equation} \label{eq:almostallc}
\ell_X(u\circ \gamma )=\int _{\gamma } f \, ,
\end{equation}
valid for $2$-almost all curves $\gamma $ in $\Omega$.
Moreover,  the restriction of $u$ to any  subdomain $O \subset \Omega$ satisfies
\begin{equation} \label{eq:confarea}
\Area (u|_{O})= \int _{O} f^2 \;.
\end{equation}

Any map $u \in N^{1,2}(D,X) $ has a well-defined trace $\trace (u) \in L^2 (S^1,X)$. If $u\in N^{1,2} (D,X)$ has a representative with a continuous extension to $\bar D$,
then $\trace (u)$ is the restriction of this extension to the boundary circle.

\section{Excluding  non-Euclidean  norms in tangent spaces}  \label{secexclude}
\subsection{Isoperimetric sets in normed planes}
Let $V$ be a $2$-dimensional normed space.  There exists a convex subset  $\mathbb{I}_V\subset V$
 with the largest possible area among all convex sets with the same  length of the boundary $\partial \mathbb{I}_V$.
 This subset is
unique  up to translations and dilations and is  called  the  isoperimetric set, \cite{Tho96}.
   The following reformulation of the Blaschke-Santalo inequality shows that
  the Euclidean isoperimetric inequality never holds in $V$ unless $V$  is Euclidean, cf. Remark \ref{rmk-ht}:
\begin{lem} \label{lem:normp}
In the notations above
\begin{equation} \label{eq:santalo}
 \mathcal H^2(\mathbb{I}_V) \geq \frac{1}{4\pi}  \ell _V^2(\partial \mathbb{I}_V),
\end{equation}
with equality if and only if $V$ is Euclidean.
\end{lem}

\begin{proof}
 After rescaling (cf. \cite[(4.10)]{Tho96}) we may assume $2 \mathcal H^2(\mathbb{I}_V) =  \ell _V(\partial \mathbb{I}_V)$.
 Then  \eqref{eq:santalo} is equivalent to $\mathcal H^2(\mathbb{I}_V)  \leq \pi$ with equality if and only if $V$ is Euclidean.
 However, due to \cite[(4.14)]{Tho96},  this is exactly the statement of the $2$-dimensional Blaschke-Santalo inequality
 \cite[Theorem 2.3.3]{Tho96}.
\end{proof}

\subsection{Formulation of the claim}
 A complete metric space $X$ has  property (ET), if for any map $u\in N^{1,2} (D,X)$ almost all approximate metric differentials
$\apmd u_z$ are Euclidean norms or degenerate seminorms,   \cite[Section11]{LW}. Examples of spaces with property (ET) are spaces with one-sided curvature bounds and sub-Riemannian manifolds. We refer to \cite{LW} for a thorough discussion of this property.
The aim of this section is to prove  the following:
\begin{thm} \label{ET}
Let $X$ be a proper metric space   with Euclidean isoperimetric inequality for curves.
Then $X$  has  property (ET).
\end{thm}

\subsection{Sobolev-Dehn function} For the limiting arguments used in the proof of \tref{ET} it is better to use a variant of the Dehn  function
with Sobolev instead of Lipschitz discs, due to better stability properties.  For a complete metric space $X$ we let the \emph{Sobolev-Dehn function of $X$}  be the minimal function $\delta _X ^{Sob } :(0,\infty)\to [0,\infty]$  for which the following holds true. For any Lipschitz curve $\gamma :S^1\to X$ of length at most  $r$ and any $\epsilon >0$  there exists a Sobolev map $u\in N^{1,2} (D,X)$ with $\trace (u)=\gamma $ and $\Area (u)\leq \delta _X ^{Sob} (r) +\epsilon$.

Since any Lipschitz disc is contained in $N^{1,2} (D,X)$ the Sobolev-Dehn function $\delta _X^{Sob}$ is bounded from above by the (Lipschitz-) Dehn function  $\delta _X$ of the space $X$.
If the space $X$ is Lipschitz $1$-connected, for instance a Banach or a
${\rm CAT}(0)$ space, then $\delta _X = \delta _X ^{Sob}$,  \cite[Propostion 3.1]{LW-asymptotic}.
For any space $X$ which satisfies the Euclidean isoperimetric inequality for curves, we have
$\delta _X ^{Sob} (r) \leq \frac 1 {4\pi} r^2$.


\subsection{Limiting arguments}
Property (ET) can be thought of as an infinitesimal property of the space,
informally expressed by the condition that tangent spaces do not contain non-Euclidean normed planes. This idea can be made precise using blow-ups of metric spaces as a special case of  ultralimits  of the rescaled original space. We refer to \cite[Section 11]{LW}  for details and just recall  the following fact:

\begin{lem} \label{ET-limit}
Let $X$ be a complete metric space and $\omega$ a non-principal ultrafilter on $\N$. Assume that for all $x\in X$ and all sequences $t_i$ of positive real numbers converging to $0$ the following holds true:  any normed plane $V$ contained in the blow-up $B=\lim_{\omega} (\frac 1 {t_n} X,x)$  is Euclidean.
Then $X$ has  property (ET).
\end{lem}

The only property of blow-ups needed for the proof of \tref{ET} is the
following stability of quadratic  isoperimetric inequalities from \cite[Theorem 1.8]{LW-asymptotic}.   Here it is crucial to work with
Sobolev-Dehn functions and properness  is used in an essential way.

\begin{lem} \label{lem:blow-up}
Let $X$ be a proper geodesic metric space with   $\delta _X ^{Sob} (r) \leq \frac {r^2}  {4\pi} $. Then for any blow-up $B$ of $X$ as in \lref{ET-limit}, we have $\delta ^{Sob} _B (r)\leq \frac {r^2} {4\pi}$ for all $r\geq 0$.
\end{lem}

\subsection{Quasi-convexity of the Hausdorff area} Using  \lref{lem:normp}  and \cite{BI} we readily obtain:
\begin{prop} \label{ET-stable}
Assume that a complete metric space $B$ contains a non-Euclidean normed
 plane $V$.  Then the Sobolev-Dehn function
 of $B$  satisfies $\delta ^{Sob} _B(r)> \frac {r^2} {4\pi}$ for all $r >0$.
\end{prop}

\begin{proof}
Let $\mathbb I_V \subset V$ be an isoperimetric set of $V$ whose boundary  $\partial\mathbb I_V$ has length $r$.
The quasi-convexity of the Hausdorff area proved in \cite{BI} together with \lref{lem:normp} implies that
 $$\delta _B^{Sob} (r) \geq  \mathcal H^2 (\mathbb I_V) >  \frac {r^2} {4\pi},$$ see also \cite[Section 2.4]{LW} and \cite[Section 10.2]{LW-intrinsic}. This finishes the proof.
  \end{proof}

  Combining \lref{ET-limit}, \lref{lem:blow-up}  and \lref{ET-stable}, we finish the proof of \tref{ET}.

\begin{rem}
A more direct but slightly more technical proof of \tref{ET} can be provided along the lines of
\cite[Theorem 5.1]{Wengerhyperbolic},  also including the case of non-proper target spaces $X$.
\end{rem}

\section{Solutions of the Plateau problem} \label{secplateau}
\subsection{Solution of the  Plateau problem}
Let $X$ be a proper metric space  with the Euclidean isoperimetric inequality for curves. Due to  \tref{ET}, the space
  $X$ satisfies  property (ET).  Let $\Gamma$ be a Jordan curve in $X$ of  finite length. Consider  the non-empty set $\Lambda (\Gamma, X)$ of all maps
$v\in N^{1,2} (D,X)$ whose trace is a weakly monotone parametrization of $\Gamma$.
A \emph{solution of the Plateau problem} for the boundary curve $\Gamma$ is a conformal map
$u\in \Lambda (\Gamma, X)$ which has smallest area among all maps in $\Lambda (\Gamma, X)$. Equivalently, $u$ is a map with minimal Reshetnyak energy $E_+^2(u)$ among all maps in $\Lambda (\Gamma, X)$, \cite[Theorem 11.4]{LW}. Due to \cite[Corollary 11.5]{LW}, a solution of the Plateau problem exists for every Jordan curve $\Gamma$ of finite length in $X$.
Any such solution of the Plateau problem  has the following property, \cite[Theorem 1.4]{LW},  \cite[Proposition 1.8]{LW-intrinsic}:

\begin{thm} \label{partI.u}
Let $\Gamma$ be a Jordan curve of  finite length in $X$ and let $u$ be a solution of the Plateau problem for the curve $\Gamma$.
Then $u$ has a  representative, again denoted by $u$, which continuously extends to $\bar D$.    For any  Jordan curve $\eta \subset \bar D$ with  Jordan domain $J\subset D$
 $$ \Area (u|_J)\leq \frac 1 {4\pi} \ell ^2 _X(u \circ \eta).$$
\end{thm}

In fact, from [LW15,Theorem 1.4] one can conclude that $u$ is locally Lipschitz on $D$.

\subsection{Intrinsic minimal disc}
Let $X,\Gamma,u$ be as in \tref{partI.u}. In \cite{LW-intrinsic} it was shown that the intrinsic pseudo-metric $d_u$ on $\bar D$ described in Subsection \ref{associated} is well defined, finite-valued and continuous with respect to the Euclidean metric.   As in Subsection \ref{associated}, denote by $Z=Z_u$ the associated metric space. Then the following holds true,  see 
\cite[Theorem 1.1, Theorem 1.2, Theorem 1.4, Theorem 1.5]{LW-intrinsic}:

\begin{thm} \label{partI.Z} Let $\Gamma$ be a Jordan curve in $X$ of finite length and let $u:\bar D\to X$ be a continuous solution of the Plateau problem with boundary $\Gamma$. Let $Z=Z_u$ be the associated length metric  space,
$P:\bar D\to Z$ the canonical projection and $\bar u:Z\to X$ the unique map with $u=\bar u \circ P$. Then we have:
\begin{enumerate}
\item  $Z$ is a  geodesic space homeomorphic to $\bar D$ and $P$ is continuous. The preimage $P^{-1} (Z\setminus \partial Z)$  is  homeomorphic to $D$.
\item The map $\bar u:Z\to X$ is $1$-Lipschitz and sends $\partial Z$ in an arclength preserving way onto $\Gamma$.
\item  For any  curve $\gamma \subset \bar D$ we have $\ell _X(u\circ \gamma )= \ell _Z(P\circ \gamma)$.
\item   For any open $V\subset D$ we have  $\Area (u|_{V})=\mathcal H^2 _Z (P(V))$.
 \item  For any Jordan curve
$\eta \subset Z$ and the corresponding Jordan domain $O\subset Z$ we have
$\mathcal H^2 (O) \leq \frac 1 {4\pi}  \ell ^2 _Z(\eta )$.
\end{enumerate}
\end{thm}

The space $Z$ in \tref{partI.Z} will be called the intrinsic minimal disc associated with $u$.

\subsection{Reduction  to \tref{thmminimal}} \label{subsec:reduct}
 Now we can prove:
\begin{prop}
\tref{thmminimal} implies \tref{thmfirst}.
\end{prop}

\begin{proof}
The "only if part" of \tref{thmfirst} has already been verified in \lref{lem:onlyif},  since the identity map $\Id :X^i\to X$ is $1$-Lipschitz and lengths of curves in $X$ and in $X^i$ coincide.   Let now $X$ be a proper metric space which satisfies the Euclidean isoperimetric inequality for curves.
 Assume in addition that any pair of points in $X$ is connected by a curve of finite length.    Consider the induced length space $X^i$. Since $X$ is proper,  the space $X^i$ is a complete geodesic metric space.  We are going to prove that $X^i$ is ${\rm CAT}(0 )$. We take an arbitrary  Jordan triangle $\Gamma \subset X^i$  and need to majorize it by some ${\rm CAT}(0)$ space
in the sense of \lref{curvmaj}.  Now $\Gamma$ has the same length when viewed as a curve in $X$.  We find a solution  $u$ of the Plateau problem for the curve $\Gamma \subset X$  and
  apply \tref{partI.u} and \tref{partI.Z} to  $\Gamma \subset X$. As in \tref{partI.Z}, we denote by $Z$ the intrinsic minimal disc associated with $u$. Thus, $Z$ is a compact geodesic metric space homeomorphic to $\bar D$  and there exists a $1$-Lipschitz  map $\bar u:Z\to X$ which maps the boundary $\partial Z$ in an arclength preserving way to $\Gamma$. Since $Z$ is a geodesic space, the map $\bar u$ considered as a map to $X^i$ is still $1$-Lipschitz and arclength preserving on
$\partial Z$. Assuming that \tref{thmminimal} holds true, the space $Z$ is ${\rm CAT}(0)$. Thus, \lref{curvmaj} implies that  $X^i$ is ${\rm CAT}(0)$ as well.
\end{proof}

\section{The conformal factor}  \label{secisop}
\subsection{An integral inequality} \label{subsecimpl}
Let $X$ be a complete metric space which satisfies the Euclidean isoperimetric inequality for curves. Let $\Gamma$ be a Jordan curve of finite length in $X$ and let $u:\bar D\to X$ be a solution of the
Plateau problem as in \tref{partI.u}.   Let  $f\in L^2(D)$ be a conformal factor of $u$.
Applying \tref{partI.u} to concentric circles and using \eqref{eq:confarea} and \eqref{eq:almostallc}
 we deduce:
\begin{lem}  \label{isop-f}
The conformal factor $f\in L^2 (D)$ satisfies the  inequality
\begin{equation}\label{eq:isop-f}
   \int_{B(z,r)}f^2 \leq    \frac 1 {4\pi} \cdot  \left(\int_{\partial B(z,r)}f\right) ^2 \, ,
 \end{equation}
  for any  $z\in D$ and almost any $0<r<1-|z|$.
\end{lem}

\subsection{Log-subharmonic functions}
Recall that a function $f:U\to [-\infty , \infty ) $ defined on a domain $U\subset \R^2$
is called \emph{subharmonic}  if  $f$ is upper semi-continuous,  contained in $L^1 _{loc}$   and satisfies
$f(z) \leq \vint_{B(z,s)}f  $ for all $z\in U$ and all $s>0$ with
$B(z,s)  \subset U$. Here and below, we denote by $\vint_T g =\vint _T  g \, d\mu $  the integral mean value $\frac 1 {\mu (T)} \int_T g \, d\mu $
 of a function integrable with respect to a measure $\mu$.
A function $f\in L^{1} _{loc} (U)$ has a subharmonic representative if and only if
$\Delta f \geq 0$ in the  distributional sense.    This representative
is uniquely defined at each point by $f(z)= \lim _{s\to 0} \vint_{B(z,s)}f $.

A function $f:U\to [0,\infty)$ is called \emph{log-subharmonic} if
$\log (f)$ is a subharmonic function. Any log-subharmonic  function is locally bounded.
Log-subharmonic functions are intimately related to non-positive curvature.

 Inequality \eqref{eq:isop-f} turns out to imply log-subharmonicity (the other implication is true as well, cf. \cite{BR}, but will not be needed here).

\begin{prop} \label{prop:subharmonic-rep}
Any non-negative function $f\in L^2 (D) \setminus \{0 \}$ which satisfies \eqref{eq:isop-f}  has a
log-subharmonic representative.
\end{prop}

\begin{proof}
We can rewrite   \eqref{eq:isop-f} in terms of the integral averages as
\begin{equation} \label{eq:classic}
\vint_{B(z,r)}f^2 \leq      \left(\vint_{\partial B(z,r)}f\right) ^2  \; .
\end{equation}
For continuous positive functions $f$  satisfying \eqref{eq:classic}, the log-subharmoni\-city  is proved in \cite[Lemma on p.~665]{BR}.
The general case  reduces to the case of smooth positive functions as follows, cf. \cite{EKMS}.
Applying  Hoelder's inequality to \eqref{eq:classic} we infer
\begin{equation} \label{eq:classic1}
\vint_{B(z,r)} f  \leq \vint _{\partial B(z,r)}  f \; .
\end{equation}
Thus, $f$ has a  subharmonic representative, see \cite[Lemma 4.6, Remark 4.8]{EKMS}. In particular, $f$ is locally bounded.

A combination
of \eqref{eq:classic} and \eqref{eq:classic1} directly implies that for any $\delta >0$ the function $f^{\delta} (z):= f(z)+ \delta$
satisfies \eqref{eq:classic} as well.   If  $f^{\delta}$ has a log-subharmonic representative for all $\delta >0$ then
 (after changing to  the subharmonic representative) we obtain $f$ as a limit of locally uniformly bounded log-subharmonic functions.
Then, by the classical convergence theorems for subharmonic functions (cf. \cite[Section 3.7,  Exercise 3.15]{Armitage}), the function  $f$ has a log-subharmonic representative.
Therefore, it suffices to prove the proposition under the assumption that $f$ is  everywhere positive.

For this, we  take  any $\epsilon >0$ and  consider the usual mollified functions $f_{\epsilon}:B(0,1-\epsilon) \to [0, \infty )$
obtained by convolutions with a standard (Friedrichs) mollifier.  We use the observation of \cite[Lemma 4.5]{EKMS} (going back to the proof in  \cite{BR}) that the smooth function $f_{\epsilon}$ still satisfies \eqref{eq:classic} for all balls contained in  $B(0,1-\epsilon)$. Due to \cite{BR}, $f_{\epsilon}$ is  log-subharmonic. By the limiting argument as above  $f$ has a  log-subharmonic representative  as well.
This finishes the proof.
\end{proof}

\subsection{Conclusion}
Taking \pref{prop:subharmonic-rep} and \lref{isop-f}   together we have shown that
the conformal factor $f\in L^2 (D)$  of our minimal disc  $u$  has a log-subharmonic  representative $\bar f$.
 From now on we will replace  $f$ by $\bar f$ and assume that $f$ is log-subharmonic.

\section{Metric defined by a conformal factor} \label{secsurf}
We refer to \cite{Reshetnyak-GeomIV} and the references therein for
a detailed description of the theory, a special case of which is sketched here.
Let $U$ be a domain in  $\R^2$ and let $f:U\to [0,\infty)$ be a log-subharmonic  function.

 For a Lipschitz   curve  $\gamma :[a,b] \to U$ define  the $f$-length of $\gamma$ to be
 $$L_f(\gamma ):=\int _{\gamma} f =  \int_a^b f(\gamma (t))\, |\gamma ' (t)|\, dt .$$
 The $f$-length does not change if $\gamma$ is reparametrized.
Since
$f$ is locally bounded the $f$-length of  any Lipschitz curve is finite.

Define $d_f: U\times U\to [0,\infty)$ by
 \begin{equation} \label{def:df}
 d_f(z, z'):= \inf \{L_f (\gamma) \;|\;  \gamma  \;\text{ Lipschitz curve between} \; z\;  \text{and} \; z'  \}.
 \end{equation}

This function $d_f$ defines a metric on $U$ and
the  identity map $i:U \to (U,d_f)$ from the Euclidean subset $U$ to the new metric space is a homeomorphism,
 \cite[Theorem 7.1.1]{Reshetnyak-GeomIV}.
Denote by $Y$ the metric space $(U,d_f)$.

  The metric $d_f$ does not  change if in \eqref{def:df} the infimum is taken over all injective curves of bounded variation of turn instead over all Lipschitz curves, \cite[p.101]{Reshetnyak-GeomIV}.  If  injective curves  $\gamma _n$  of uniformly bounded variation of turn converge pointwise to the curve  $\gamma$ in $U$ then,  due to  \cite[Theorem 8.4.4]{Reshetnyak-GeomIV},
  \begin{equation} \label{stadler}
 \lim L_f (\gamma _n) = L_f (\gamma ) \,.
 \end{equation}
 Thus,  the distance $d_f  (z_1,z_2)$ can be  defined by the formula
\eqref{def:df}, where the infimum is taken over the set of all polygonal curves $\gamma \subset U$ between $z_1$ and $z_2$.

 For any  compactly contained subdomain $V\subset U$, the restriction $i:V\to Y$ is Lipschitz continuous, since $f$ is locally bounded.
  For any Lipschitz curve $\gamma$  in $U$ we have $\ell _Y(i\circ \gamma )= L_f(\gamma )$,
\cite{Resh-length}. We deduce from \eqref{eq:almostall} that $i:V\to Y$ is conformal with conformal factor $f$.  Therefore:
\begin{equation}  \label{eq:confarea+}
 \mathcal H^2 (i(V)) = \int _V f^2 \; .
\end{equation}

The main results of Reshetnyak's analytic theory of Alexandrov surfaces of (integral) bounded curvature \cite[Theorems 7.1 and 7.2]{Reshetnyak-GeomIV} take in our case  the following form:

\begin{thm} \label{thm:Resh}
The space $Y=(U,d_f)$ constructed above has non-positive curvature. Conversely, for any space $M$ of non-positive curvature  which is homeomorphic to a surface without boundary the following is true. For any point $x\in M$ there exists a neighborhood of $x$ isometric to some  $Y=(U,d_f)$, where $U$ is a domain in $\R^2$ and $f:U\to [0,\infty) $ is  log-subharmonic.
\end{thm}
\begin{proof}
We merely explain why  \tref{thm:Resh} is a special case of Reshetnyak's results, relying on the results presented in \cite{Reshetnyak-GeomIV}.

Recall that a locally compact length space $X$  has non-positive curvature if and only if any point $x\in X$ has a neighborhood $U$ such that
any triangle $\Gamma$ in $U$ has  a non-positive excess, \cite[p.36]{Aleksandrov}. (Here  the excess of a triangle is the sum of its angles minus $\pi$.) Thus  a length space $X$ homeomorphic to a surface without boundary has non-positive curvature if and only if it has
bounded (integral)  curvature in the sense of Aleksandrov and the (signed) curvature measure of $X$ is non-positive.

Now both claims of \tref{thm:Resh} are exactly Theorems 7.1 and  7.2 in \cite{Reshetnyak-GeomIV}.
\end{proof}

\section{Reduction to \tref{thm4}} \label{sec:mainred}

\subsection{Formulation} The aim of this section is  to prove:

\begin{prop} \label{prop:reform}
\tref{thm4} implies \tref{thmminimal}.
\end{prop}

Thus we assume that \tref{thm4} is true. Let $X$ be a proper metric space which
satisfies the Euclidean isoperimetric inequality for curves. Let
$\Gamma$ be a Jordan curve of finite length in $X$. Let $u:\bar D\to X$ be a solution of the Plateau problem in $X$ for the boundary curve $\Gamma$.
Denote by $Z$ the intrinsic minimal disc associated with $u$ as in \tref{partI.Z} and let $P:\bar D\to Z$ be the canonical surjective map.
Let $Z_0$ be the open disc $Z\setminus \partial Z$ and denote by $D_0$ the preimage $P^{-1}(Z_0)$. Then $D_0$ is homeomorphic to the open disc and, in particular, $D_0\subset D$.

Let $f$ denote the conformal factor of $u$, which is log-subharmonic due to Section \ref{secisop}.
Denote by $Y_0$ the open disc $D_0$ equipped with the length metric $d_f$ as introduced  in the previous section.  \tref{thm:Resh} implies that $Y_0$ has non-positive curvature.
Let $i:D_0\to Y_0$ denote the canonical homeomorphism (identity map).  Let $I:Y_0\to Z_0 \subset Z$
denote the composition $P\circ i ^{-1}$.
 We now easily conclude, using \tref{thm4}:

\begin{lem}
If $I:Y_0\to Z_0$ is a local isometry  then $Z$ is ${\rm CAT}(0)$.
\end{lem}

\begin{proof}
If $I$ is a local isometry then it is locally injective.  By the invariance of domains we  see that $I$ is an open map.
  Since $I$ is surjective and $Y_0$ non-positively curved, the space  $Z_0$ has non-positive curvature.
  Due to \tref{partI.Z}, the space $Z$ satisfies
the isoperimetric inequality for Jordan curves as required in
  \tref{thm4}, (3).  From \tref{thm4}  we deduce that $Z$ is ${\rm CAT}(0)$.
\end{proof}

Therefore, in order to prove \pref{prop:reform}, we only need to show that $I:Y_0\to Z_0$ is a local isometry.

\subsection{Properties of the map $I$}
We claim:
\begin{lem} \label{1-lip}
The map $I:Y_0\to Z_0$ is $1$-Lipschitz.
\end{lem}

\begin{proof}
Since the metric in $Y_0$ can be defined using $f$-lengths of polygonal curves  we only need
to prove $L_f(\gamma )\geq  \ell_Z (P\circ \gamma)=\ell _X (u\circ \gamma   )$ for any  straight segment   $\gamma:[a,b]\to D$. Consider
the variation $\gamma _s$, for $-\epsilon <s< \epsilon $, of $\gamma _0=\gamma $ through segments parallel to $\gamma $.
For any $s$, we have $L_f(\gamma _s) =\int _{\gamma _s} f$. Moreover, by the definition of the conformal factor,
we have $\ell _X(u\circ \gamma _s)= \int _{\gamma _s} f$ for \emph{almost all} $s\in (-\epsilon, \epsilon)$.
Thus, we find a sequence $s_n \to 0$ such that $L_f(\gamma _{s_n})= \ell _X (u\circ \gamma _{s_n} )$.
The result follows from
$$  L_f(\gamma ) =\lim_{n\to \infty}  L_f(\gamma _{s_n})=\lim_{n\to \infty} \ell _X (u\circ \gamma _{s_n} ) \geq \ell _X (u\circ \gamma _{0} ) =\ell _X(u\circ \gamma ),$$
 where we have used  \eqref{stadler} for the first equality.
\end{proof}

 By construction, \eqref{eq:confarea}, \eqref{eq:confarea+} and \tref{partI.Z} (4),  we obtain:
\begin{lem}
The map $I$ preserves the Hausdorff measure $\mathcal H^2$. More precisely,
for any open subset $V\subset D_0$ we have
$$\mathcal H^2 (i(V))  =\int _V f^2 =\mathcal H^2 (P(V)) .$$
\end{lem}

\subsection{The conclusion} Now we can finish the proof of the main result of this section.
\begin{proof}[Proof of  \pref{prop:reform}]
By the definition of the metrics on $Y_0$ and $Z$
 it suffices to show that
for any simple  curve $\eta:[a,b] \to Y_0$ we have  $\ell _{Y_0}(\eta)=\ell _Z(I\circ \eta)$.
 Assume the contrary and consider  a curve $\eta$ with
$$\ell _{Y_0}(\eta)>\ell _Z(I\circ \eta).$$

In order to obtain a contradiction we will roughly proceed as follows. We will first complement a subcurve of $\eta$ to a Jordan curve and equip the closure of the corresponding Jordan domain with a new metric, to which \tref{thm4}  will be applied. Inside the domain, the new metric will come from that of $Y_0$, while the length of the boundary will come from that of its image in $Z$.

 To be more concrete,  we first replace the curve $\eta$  by a  subcurve and may assume that either  $I\circ \eta $ is a point or that no subarc of $\eta$ is mapped by $I$ to a point.  Further replacing $\eta$  by a slightly smaller subcurve, we may assume that $\eta$ is part of a Jordan curve $T$ such that the closure of $T\setminus \eta$ is a
rectifiable  arc $\eta'$ in $Y_0$. Let $J\subset Y_0$ denote the Jordan domain of $T$ whose closure is $\bar J=J\cup T$.


 We  call admissible any curve in $\bar J$ that is a  finite concatenation of simple curves
either completely contained in $\eta $ or intersecting $\eta$ at most  in  a finite set of points.  Define the new length $\mathcal L^+ (\gamma) $ of such an admissible  curve $\gamma$ to be the sum of the $Y_0$-lengths of arcs outside of $\eta$ and the lengths of the images in $Z$ of the subarcs contained in $\eta $.
 The   pseudo-distance  $d^+ :\bar J\times \bar J\to [0,\infty] $  associated with the length functional $\mathcal L^+$ is finite-valued and continuous with respect to the Euclidean topology on $\bar J$ since $\eta '$ and $I\circ \eta$ have finite length.
Denote the corresponding metric space  by $S$ and let $Q:\bar J\to S$ be the canonical projection.

Then $S$ is a compact length space, hence it is a geodesic space.
The map $Q$ is a local isometry outside  $T$ and a homeomorphism outside $\eta$ (here and below we consider $\bar J$ with the metric restricted  from $Y_0$, not with a length metric!).
If $I$ does not send $\eta$ to a point (hence is bijective by assumption) then $Q$ is bijective, hence a homeomorphism. If $I$ sends $\eta$ to a point then $Q$  collapses $\eta$ to a point in $S$. In both cases, $S$ is homeomorphic to
$\bar D$.   The restriction $I:\bar J\to Z$ factorizes through $Q$ and for any curve $k\subset \bar J$ we have $\ell _S (Q\circ k) \geq \ell _Z (I\circ k)$ by \lref{1-lip}.


 We claim that $S$ is ${\rm CAT}(0 )$.  Indeed, $S\setminus \partial S$ is locally isometric to $J$. Thus $S\setminus \partial S$ has non-positive curvature. Due to  \tref{thm4} we only need to verify  the isoperimetric inequality for Jordan curves $c\subset S$.
  For any Jordan curve $c$ in $S$ there is  a unique Jordan curve
$\hat c$ in $\bar J$ which is mapped by $Q$ to $c$.  Let $G\subset J$ denote the Jordan domain of $\hat c$. Then $Q(G)\subset S$ is the Jordan domain of $c$ and has the same Hausdorff measure as $G \subset Y_0$ since $Q$ is a local isometry outside   $T$.

We have:
$$\ell _S(c) =\ell  _S(Q\circ \hat c) \geq \ell _Z (I\circ \hat c)=\ell _X (u\circ i^{-1} \circ \hat c) \; .$$
 Moreover,
 $$\mathcal H^2 _{S} (Q(G)) = \mathcal H^2 _{Y_0} (G) =\mathcal H^2 _Z (I (G)) =\Area (u| _{i^{-1} (G)}) \;.$$
 Since $i^{-1} (G) \subset D$ is the Jordan domain of $i^{-1} (\hat c)$, the desired inequality
 $$\mathcal H^2  _S (Q(G)) \leq \frac 1 {4\pi}  \cdot \ell ^2 (c)$$
 follows from  \tref{partI.u}.  Therefore, an application of \tref{thm4}  finishes the proof of the claim that $S$ is ${\rm CAT}(0 )$.

 Another application of \tref{thm4} implies that the metric on $S\setminus \partial S$ is a length metric. Therefore, the length preserving map
 $Q^{-1} :S\setminus \partial S\to J$ is $1$-Lipschitz.
 Thus, it extends to a $1$-Lipschitz map $\hat Q ^{-1}:S\to \bar J$. By continuity, the composition $\hat Q^{-1} \circ Q$ must be the
identity on $\bar J$.

By assumption, $\ell _S(Q\circ \eta) <\ell _{Y_0} (\eta)$.
We obtain a contradiction to  $\eta =\hat Q^{-1}\circ Q\circ \eta$, since the $1$-Lipschitz map $\hat Q^{-1}$ cannot increase lengths.
\end{proof}

Thus Theorems \ref{thmfirst} and \ref{thmminimal} are reduced to \tref{thm4}.

\vspace{3cm}

\centerline{\bf{Part II. Geometry of strange metric discs.}}

\vspace{0.5cm}

\section{Plan of the second part of the paper} \label{Intro2}
 The second part of the paper  is devoted to the proof of \tref{thm4}.
Before we start with the actual proof of the theorem, we recollect
in Section \ref{sec:basicgeom}  some basic observations about the local geometry of non-positively curved surfaces.
In Section \ref{seccomplete} we prove that the  completion of a non-positively curved open disc  is ${\rm CAT} (0)$ and    homeomorphic to a closed disc whenever compact.
 All but the main implication (3) to (1) of \tref{thm4} turn out to be relatively simple. The proofs of these simple implications are provided in  Section \ref{sec:boundary}.

 In Section \ref{sec:simple} we embark on the proof of the implication (3) to (1) of \tref{thm4}, thus assuming the isoperimetric inequality and trying to prove that the disc $Z$  is ${\rm CAT} (0)$.  We use subdomains of $Z$ in order to reduce the situation to the case where the "problematic" part of $\partial Z$ consists of a geodesic $c\subset \partial Z \subset Z$.
  We consider the complement $Y_0 =Z\setminus \partial Z$ equipped with  the induced length metric and the completion $Y$ of $Y_0$, which is ${\rm CAT}(0)$ and compact. The space $Y$ comes along with a canonical $1$-Lipschitz map $I:Y\to Z$ which is a local isometry in $Y_0$.
This reduces the question to the situation described in the introduction (the example of Koch's snowflake):  the disc $Z$ arises from
a ${\rm CAT}(0 )$ disc $Y$ by possibly shortening or collapsing a part of the boundary curve $\eta \subset Y$.

If $I$ is  an isometry then we are done. Otherwise, some part $\eta$ of the boundary curve $\partial Y$ is sent by $I$  to a curve in $\partial Z$ of smaller length.  In order to obtain a contradiction to the isoperimetric inequality, it  suffices to extend small parts of $\eta$ to Jordan curves $T$  in $Y$ which bound almost optimal isoperimetric regions  in $Y$.  The map $I$ then shortens the length of $T$ but leaves the area of
the enclosed Jordan domain unchanged, thus $I (T) \subset Z$  violates (3) of \tref{thm4}.

If $\eta$ is rectifiable we approximate $\eta$ by geodesics and use the approximation of $Y$ by flat cones in order to reduce the problem to the  situation where $\eta$ is a line in $\R^2$. In that case one can take  as suitable Jordan curves $T$ large parts of sufficiently  large circles   complemented by a  short  chord contained in  $\eta$.  This is carried out in Section \ref{sec:rect}.
The more difficult  case of a non-rectifiable curve $\eta$ is carried out in Section \ref{sec:final}. Here the non-rectifiability of $\eta$ provides parts which are contracted by an arbitrary large amount. This allows sufficient flexibility in the choice of critical Jordan curves in $Y$.

\section{Geometry of non-positively curved  surfaces} \label{sec:basicgeom}
\subsection{Basic geometric features}
Let $M$ be a metric space  of non-positive curvature  homeomorphic to $D$.   We refer to \cite{Reshetnyak-GeomIV} and \cite{BB} for a deep analysis of such and related  more general spaces.
 Let $x\in M$ be arbitrary.
We find a small open metric ball  $U=B(x,\epsilon) $ around $x$ whose closure  is a  compact ${\rm CAT}(0 )$ space. Any geodesic starting at $x$ can be extended
 to a geodesic in $\bar U$  of  length $2\epsilon$,  \cite[1.B.7]{BB}.  The space of directions $\Sigma _x$ is a circle of some length $\alpha \geq 2\pi$. By definition, the tangent cone $T_xM$ at the point $x$ is the Euclidean cone over $\Sigma _x$.

\subsection{Hinges} Let $M, x\in U=B(x,\epsilon) \subset  M$ be as above and let $\gamma _{1,2}$ be geodesics of length $\geq \epsilon$ starting in $x$ and having only the point $x$ in common. 
Then $\gamma _{1,2}$ intersects the boundary circle of $B(x,\epsilon)$  at exactly one point.
Denote by $\Gamma$ the union of (the images of)  $\gamma _1$
and $\gamma _2$ inside $B(x,\epsilon )$. Then $\Gamma$ divides $B(x,\epsilon)$ into two closed subsets $H_{\pm}$ homeomorphic to closed half-planes  intersecting in their common boundary  $\Gamma$.
We  call  $H_{\pm}$,  equipped with the induced length metric,  the \emph{hinges} (of size $\epsilon$) defined by $\gamma _{1,2}$.

We claim that both hinges have non-positive curvature.  If $\gamma_1$ and $\gamma _2$ concatenate to a  geodesic then $H_{\pm}$ are  convex in $B(x,\epsilon)$ and the claim follows. Otherwise, we extend $\gamma _1$ by a geodesic $\gamma _1 ^+$ of length $\epsilon$ starting in $x$ to a geodesic $\gamma$ of length $2\epsilon$. Then $\gamma$ divides $B(x,\epsilon)$ into two convex subsets  $A_{\pm}$  with common boundary $\gamma$.
  Without loss of generality, we may assume that  $H_+$ is contained in $A_+$.
  Then $\gamma \cup \gamma _2$ divide $B(x,\epsilon)$ into $3$ convex hinges $H_+$, $A_-$ and
  a third hinge $H'$ (between $\gamma _2$ and $\gamma _1^+$).  The hinge
  $H_+$ has non-positive curvature  by convexity and $H_-$ is the result of gluing  $H'$ and
  $A_-$ along the geodesic $\gamma _1^+$, hence it has non-positive curvature
   by Reshetnyak's gluing theorem. This finishes the proof of the claim.

\begin{lem} \label{lem:Jordomloc}
 Let $M$ be a space of non-positive  curvature homeomorphic to the open disc. Let
 $\Gamma$ be a Jordan polygon in $M$. Then the closed Jordan domain
$\bar J$  of $\Gamma$ with its intrinsic metric is ${\rm CAT}(0)$.
\end{lem}

\begin{proof}
The space $\bar J$ is compact and simply connected. At any point $x\in \bar J$ a small ball around $x$ is either open in $M$ or isometric to a hinge described above.  Therefore,
$\bar J$ is non-positively curved. The lemma follows from the theorem of
Cartan-Hadamard.
\end{proof}

\subsection{Approximation by flat cones}
For the next result we will use \tref{thm:Resh} from above
and a theorem about the approximation of Alexandrov surfaces by their tangent cones.

\begin{lem} \label{almostcone+}
Let $X$ be a space of non-positive curvature   which is homeomorphic to a surface (possibly with boundary). Let
$x\in X$ be a point. Let $\gamma_{1,2}$  be geodesics starting at $x$ and enclosing a positive angle. Then there exists a closed
interval $T\subset \R$, a ball $O$ around the vertex  $o$ of the Euclidean cone $CT$ over $T$ and a biLipschitz map $E:O\to E(O) \subset X$ with the following properties.
\begin{enumerate}
\item The map $E$ sends the vertex $o$ of $CT$ to $x$.
\item  $E$ sends  initial parts of the boundary rays of the flat hinge $CT$ isometrically
onto the initial parts of $\gamma _{1,2}$.
\item  We have $\lim _{v,w \to o} \frac  {d(E(v),E(w))}  {d(v,w)}\to 1$.
\end{enumerate}
\end{lem}

Thus, the biLipschitz constant  of the restriction  of $E$ to small balls around $o$ goes to $1$ with the radius of the balls tending to $0$.
Clearly,  the length of the interval  $T$, hence the total angle of the Euclidean hinge $CT$,
is not less than the angle between $\gamma _1$ and $\gamma _2$.

\begin{proof} [Proof of  \lref{almostcone+}] We first  assume   that $x$ is not on the boundary $\partial X$.
The existence of a biLipschitz map $E':O\to M$ from a ball $O$ around the origin $0\in T_xM$ such that
(1) and (3) of \lref{almostcone+} hold true is the content of a theorem of Y. Burago, \cite{Burago},  stated in \cite[Theorem 9.10]{Reshetnyak-GeomIV}. Note that this theorem is applicable due to \tref{thm:Resh} above.
Composing $E'$ with a self-isometry of $T_xM$ we can assume that the initial part of a given ray $\eta$ starting at the origin $0$ is sent by $E'$ to a curve $E'\circ \eta $ whose starting direction coincides with the starting direction of $\eta$.  Consider the rays $\eta_{1,2}$ in $T_xM$ whose starting directions are $\gamma'_{1,2} (0) \in \Sigma _x\subset T_xM$. We now find a biLipschitz map of  $T_xM$ to itself, which fixes $0$, has at  $0$ the identity as its differential, and which  sends the initial part of $\eta _i$ to  $(E')^{-1} (\gamma _i)$. Composition of $E'$ with this biLipschitz map provides the required map $E$ upon restriction to the smaller hinge in $O$ between the rays $\eta_i$.

 Let us now assume that $x$ is contained in the boundary $\partial X$. We choose any simple arc $\gamma$
connecting a point on $\gamma _1$ with a point on $\gamma _2$ and such that $\gamma$ and the corresponding parts of $\gamma _1$ and $\gamma _2$ constitute together a Jordan curve $T$. Consider the union   $V$ of the corresponding Jordan domain and the curve  $T \setminus \gamma$. Since $\gamma _{1,2}$ are geodesics and $x\in \partial X$, the set $V$   is locally convex in $X$ by topological reasons.   Thus, in order to find a biLipschitz embedding required in \lref{almostcone+}   we may replace $X$ by $V$ and therefore assume
that $\partial X$ is  the union of the geodesics $\gamma _{1,2}$. But (a small ball around $x$ in) such a space is isometric to a hinge in some  manifold $M$ \emph{without boundary} which still has non-positive curvature, as we see by  applying  Reshetnyak's gluing theorem twice.  (First we glue to $V$ along $\gamma_1$ a  hinge     of large angle from $\R^2$. In the so arising space the boundary is a geodesic and we may double the arising space to obtain the required manifold $M$).  By construction $V$ is the smaller hinge of the two hinges determined by $\gamma _{1,2} \subset M$.  Thus, the map $E$ constructed above (for the space $M$) has its image in $V$.
This finishes the proof of \lref{almostcone+}.
\end{proof}

Using the notations of \lref{almostcone+}, we will call {\it hinge}  between two geodesics $\gamma _{1,2}$ the intersection of $E(O)$ with a small metric ball $B (x,r)$.   This provides an extension of the definition of a hinge to the
case of points at the boundary.

\section{Completions of $2$-dimensional  open discs} \label{seccomplete}
Before embarking on the proof of \tref{thm4} we will study the interiors of discs appearing in that theorem and their completions.
The following  basic result generalizes \cite{bishop}.

\begin{prop} \label{bisgen}
Let $Y_0$ be a length space homeomorphic to the open disc.  If $Y_0$ is non-positively curved then the completion $Y$ of $Y_0$ is ${\rm CAT}(0)$.
\end{prop}

\begin{proof}
Choose Jordan curves  $\Gamma _n'$  with increasing Jordan domain whose union is $Y_0$.  Approximate $\Gamma _n'$
by Jordan polygons $\Gamma _n$.    Let $J_n$ be the corresponding  Jordan domains and denote by $Y_n$ the closure
of these Jordan domains, equipped with their intrinsic metric.

By \lref{lem:Jordomloc}, every $Y_n$ is ${\rm CAT}(0)$.
The completion $Y$ isometrically (and canonically) embeds into the ${\rm CAT}(0 )$ space
$Y'$, obtained as  an ultralimit of the $Y_n$ (choosing the same fixed point lying in $Y_1\subset Y_n$  as the base point of the spaces $Y_n$).  Hence $Y$ is  isometric to a subset of the
 ${\rm CAT}(0 )$ space $Y'$. Due to completeness,   $Y$ is a closed subset of $Y'$.
Since  $Y_0$ is a length  space, its completion $Y$ is a length space as well. Therefore, $Y$  must be convex in $Y'$.  Thus  $Y$ is  ${\rm CAT}(0 )$.
\end{proof}

 The reader should consult    \cite[p.1270]{anguel}  for references  on homology manifolds used in  the next lemma.

\begin{lem} \label{lem:topology}
Let $Y_0$ be a non-positively curved length metric space homeomorphic to $D$.  If the completion $Y$ of $Y_0$ is compact
 then $Y$  is homeomorphic to $\bar D$.
\end{lem}

\begin{proof}
By construction, $Y_0$ is dense in $Y$. Since $Y_0$ is locally complete, $Y_0$ is open in $Y$.
 The space $Y$ is a separable ${\rm CAT}(0)$ space, hence it is contractible and locally contractible. The topological dimension of $Y$ coincides with its geometric dimension, see \cite{Kleiner}.   But $Y$  embeds isometrically into an ultralimit of $2$-dimensional  ${\rm CAT}(0)$-spaces, therefore the geometric dimension of $Y$ is
at most $2$, see \cite{Kleiner},    \cite[Lemma 11.1]{Lytb}.  Thus $Y$ has topological dimension $2$.

 Set $\partial Y := Y\setminus Y_0$.  We claim that for any  $z\in \partial Y$ the local homology with integer coefficients $H_{\ast} (Y,Y\setminus \{z \})$ vanishes. In order to see this,
 note that $Y\setminus \{z  \}$ is connected  since $Y_0$ is connected.
 By dimensional reasons, the contractibility of $Y$, and  the long exact sequence of the pair $(Y,Y \setminus \{z\})$,  we only need to prove that
  $H_1 (Y\setminus \{z\} ) =0$. Since $Y$ is locally contractible and $Y_0$ is contractible it is sufficient to prove that any
  closed curve $\gamma :S^1\to Y\setminus \{z \}$ can be approximated  by closed curves with images in $Y_0$.  Covering $\gamma$ by small metric balls $B(x,r)$, we observe  that it is sufficient to prove that $B(x,r)\cap Y_0$ is connected for any $x\in Y$ and $r>0$. But this  follows from the fact that $Y_0$ is a length space and $Y$ is the completion of $Y_0$.
   This finishes the proof of the claim.

 Thus $Y$ is a homology $2$-manifold with boundary $\partial Y$.   Therefore,  $ \partial Y$
  is a homology $1$-manifold and, due to
\cite{Raymond},  the doubling  $Y^+$ of $Y$ along the boundary $\partial Y$ is a homology $2$-manifold without boundary.
Due to \cite[IX.5.9 and IX.5.10]{Wilder}, for $n=1,2$, any homology
$n$-manifold without boundary is a manifold without boundary.  Thus  $Y^+$ is  $2$-manifold  and $\partial Y$ is a closed $1$-submanifold.
We deduce that $Y$ is a manifold with boundary. Since $Y_0$ is homeomorphic to $D$, the space $Y$
must be homeomorphic to $\bar D$.
\end{proof}

\section{Simple implications} \label{sec:boundary}
We now embark on the proof of \tref{thm4}. Thus, from now on, let $Z$  be  a geodesic metric space homeomorphic to $\bar D$ and such that
the space $Z\setminus \partial Z$ is non-positively curved.

\subsection{(1) implies (3)} \label{1to3}
Suppose  the space $Z$ is ${\rm CAT}(0 )$ and consider a Jordan curve $\eta\subset Z$ of finite length $\ell$.  Then Reshetnyak's majorization theorem (\lref{lem:onlyif}) provides a Lipschitz disc $u:\bar D\to Z$ filling $\eta$ of area  at most $\frac  1 {4\pi}\cdot \ell ^2$.
For topological reasons, the image of $u$ must contain the Jordan domain $J$ of $\eta$. Therefore,
$\mathcal H^2 (J)\leq \Area (u)  \leq \frac  1 {4\pi}\cdot \ell ^2.$

\subsection{(2) implies (1)} \label{2to1}
Note that $Z$ is the completion of  $Z\setminus \partial Z$. Thus, if $Z\setminus \partial Z$ is a length space then $Z$ is ${\rm CAT}(0)$ by \pref{bisgen}.

\subsection{(1) implies (2)} \label{1to2}
Thus, we assume that $Z$ is ${\rm CAT}(0)$
and claim that $Z\setminus \partial Z$ is a length space.
The proof of this implication (probably  well-known to experts)  is slightly more involved.

Consider arbitrary points $y_+,y_-\in Z\setminus \partial Z$. Let $\gamma: [a,b] \to Z$ be a geodesic between $y_+$ and $y_-$.
Fix a positive $\epsilon >0$. We need to find a curve $\gamma _{\epsilon}$  in $Z\setminus \partial Z$ which connects $y_+$ and $y_-$ and has length at most $(1+2\epsilon) \cdot \ell (\gamma) $.

For any  $y=\gamma (t)  $ we  claim the existence of an open  ball $W^t$ around $y$ with the following property. For any $z_1,z_2 \in W^t \cap \gamma  $ there exists a curve $\eta \subset W^t$ connecting $z_1$ and $z_2$, with $\eta \cap \partial Z \subset  \{z_1,z_2 \}$ and  such that
$$(1+\epsilon) \cdot d_Z(z_1,z_2)\geq \ell (\eta).$$
Indeed,   for $y \notin \partial Z$, in particular for $y=y_{\pm}$,  the claim is evident. For any $y\in \partial Z$  we apply
\lref{almostcone+} to  both parts of $\gamma$ emanating from $y$ and deduce the claim from the corresponding result in the flat hinge $CT$, where the claim is clear as well.

We can cover $\gamma $ by a finite number of open balls $U_i=W^{t_i}$.
 By choosing an appropriate subsequence and rearrangement, we may assume
   that any two consecutive balls in the sequence $U_i$ intersect.   Choose arbitrary  points $y_i$ on $\gamma$ in the intersection of $U_i$ and $U_{i+1}$, with the only requirement that $y_i \notin \partial Z$ whenever $U_i\cap U_{i+1} \cap \gamma$ is not completely contained in $\partial Z$.
    We connect
   $y_i$ and $y_{i+1}$  by a  curve $\eta_{i+1} \subset U_{i+1}$ provided by the definition of $W^{t_i}$. We may assume   the first $y_0$ and the last $y_m$ to be the  ends  $y_{\pm}$ of $\gamma$.  Denote  by $\eta $ the concatenation of all $\eta _i$. Then $\eta$ is a curve in $Z$ between $y_+$ and $y_-$  which has length at most $(1+\epsilon) \cdot \ell (\gamma) $. Moreover, $\eta$ intersects $\partial Z$ at most at finitely many points $y'_j=\eta (s_j)$.

For any such $y'_j$,  a neighborhood of $y'_j$ in $\partial Z$ is contained in $\gamma$.
We again apply \lref{almostcone+} to  both parts of $\gamma$ emanating from $y'_j$ and note that in this case the image of the map $E$ must be \emph{open} in $Z$ by the invariance of domains. Using the corresponding property in the flat cone
$CT$, we find an arbitrarily short curve $k_j $ with the following property. The curve $k_j$ connects two points  $\eta (a_j),\eta (b_j)$ with some $a_j<s_j<b_j$ and  does not intersect $\partial Z$. Replacing $\eta |_{[a_j,b_j]}$ by  $k_j$ we obtain the desired short connection between $y_+$ and $y_-$.

\section{Some simplifications} \label{sec:simple}
The rest of the paper is devoted to the proof that (3) implies (1) in \tref{thm4}. Thus,  from now on we assume that $Z$  is a geodesic metric space homeomorphic to $\bar D$, such that $Z\setminus \partial Z$ is non-positively curved and such that $Z$  satisfies the isoperimetric inequality as stated in \tref{thm4}, (3).  We need to prove that $Z$ is ${\rm CAT}(0)$.

\subsection{Subdomains}
Let $T$ be  a Jordan curve  of finite length in $Z$ with Jordan domain $J$. Consider the closure $\bar J =J\cup T\subset Z$ with the induced length metric.    Since $T$ has finite length, the topologies induced by the length metric and by the induced metric coincide. Thus $\bar J$ is a compact length metric space homeomorphic to $\bar D$. The compactness implies that $\bar J$ is geodesic. The identity map $F:\bar J\to Z$ preserves the lengths of all curves and the $\mathcal H^2$-area of all domains.  Moreover, the restriction of $F$ to $J=\bar J \setminus T$  is a local isometry. Therefore, the assumptions of \tref{thm4}, (3)  are valid for the space $\bar J$ as well.
As a consequence, we may reduce the problematic part of $Z$ to a single geodesic:

 \begin{lem}  \label{lem:simp}
 We may assume in addition that
 \begin{enumerate}
 \item  $\mathcal H^2 (Z)$ is finite.
 \item   There is a geodesic $c:[a_0,b_0]\to \partial Z \subset Z$  and some $a_0<a<b<b_0$,
such that  $Z\setminus c ([a,b])$ has non-positive curvature.
 \end{enumerate}
 \end{lem}

\begin{proof}
Assume  some $Z$ satisfies the assumption of \tref{thm4},(3) but is not ${\rm CAT}(0)$. We are going to find a Jordan domain $Z^-$ in $Z$ which satisfies both assumptions of the lemma,  but which is not ${\rm CAT}(0)$ either.


 Due to Subsection \ref{2to1},  $Z \setminus \partial Z$ cannot be  a  length space. Thus we find points  $x,y \in  Z \setminus \partial Z$ and
 $\epsilon >0$  such that
  for any curve $\gamma \subset Z\setminus \partial Z$ connecting $x$ and $y$ we have $\ell (\gamma ) > d_Z(x,y) +\epsilon$.

  Consider a geodesic $c:[a_0,b_0]\to Z$ in $Z$ between $x$ and $y$.
  Connect  further $x$ and $y$ by some simple piece-wise geodesic  curve  $\hat c$ in $Z \setminus \partial Z$  disjoint from $c$ outside the endpoints.  Consider the arising Jordan curve $T =c\cup \hat c$  and
 the corresponding  closed Jordan domain $Z^-$ with its induced length metric. Since $T$ has finite length we have  $\mathcal H^2 (Z^-)<\infty$ by the isoperimetric inequality.
     Outside the intersection of $T$ with $\partial Z$, the space $Z^-$ has non-positive curvature  by    \lref{lem:Jordomloc}.
  Thus $Z^-$  satisfies (1) and (2) required in the lemma, where $a<b$ can be chosen arbitrary in $(a_0,b_0)$, sufficiently close to $a_0$ and $b_0$,
  respectively.

  Assume that $Z^-$ is  ${\rm CAT}(0 )$. Then $Z^-\setminus T$, with the metric induced from $Z^-$, is a length space due to
  Subsection \ref{1to2}.
     Therefore, we find a curve $\gamma $ in $Z^-$ connecting $x$ and $y$, such that $\gamma$ does not intersect $T\cap \partial Z$ and such that the length of
  $\gamma$ is arbitrary close to  $\ell (c)= d_Z(x,y)$. This contradicts our assumption on $x$ and $y$, shows that $Z^-$ cannot be ${\rm CAT}(0 )$,
  and finishes the proof of the lemma.
\end{proof}

\subsection{Simple setting} We may and will assume from now on that $Z$ satisfies the assumption of \lref{lem:simp}. Thus
 $Z$ has non-positive curvature  outside a geodesic $c$ contained in $\partial Z$.
We consider the complement $Z\setminus \partial Z$  and call the associated length space $Y_0$.  The identity map $I:Y_0\to Z\setminus \partial Z$ is a $1$-Lipschitz homeomorphism.  The length space  $Y_0$ has non-positive curvature, thus,
by \pref{bisgen}, the completion $Y$ of $Y_0$ is ${\rm CAT}(0)$. Moreover,  $I:Y_0\to Z$ extends to a $1$-Lipschitz map $I:Y\to Z$.
Under these assumptions we obtain the following  uniform area estimate for balls in $Y_0$ near $c$.

\begin{lem}  \label{lem:areabound}
For any
$y\in Y_0$ and all $r$ such that  $\overline{I(B_{Y_0}(y,r))}\cap \partial Z \subset c$,
  the area of the $r$-ball in $Y_0$ around $y$ satisfies:
$$\mathcal H^2 (B_{Y_0} (y,r))\geq  \frac {\pi} 4 r ^2 \; .$$
\end{lem}

\begin{proof}
 We  proceed similarly to  \cite[Section 9.2]{LW-intrinsic}, following the standard arguments  leading to the   boundary regularity
 of minimal surfaces under a chord-arc condition on the boundary.

Let $y\in Y_0$ and  $r>0$ be as in the statement  of the lemma.
Since $\partial Z$ is not completely contained in $c$, the assumption on $r$ implies that there exists some "remote" point  $x\in Y_0$ with the following property: $d(x,y)>r$ and   for the connected component $U$ of $x$ in $Y_0\setminus B(y,r)$ the closure of $I(U)$ in $Z$ contains a point in $\partial Z\setminus c$.

 We now argue by contradiction and assume that $\mathcal H^2 (B_{Y_0} (y,r))< \frac {\pi} 4 r ^2 $.
  For any  $t<r$, we consider the ball $B_{Y_0}(y,t)$ around $y$ in  $Y_0$, denote by $v(t)$ its area and by   $L_t$ its boundary in $Y_0$.  Note that $L_t$ separates $x$ and $y$ in $Y_0$. Set $w(t)=\mathcal H^1 (L_t)$.
 By the co-area inequality  we have  $\int _0 ^t w(s) \; ds \leq  v(t)$ for all $t<r$. (See \cite[Lemma 2.3]{LW-intrinsic} and note that
 $Y_0$ is  countably $2$-rectifiable,   no non-Euclidean planes appear as tangent spaces of $Y_0$ and the distance function to any point is $1$-Lipschitz).
  The contradiction to $v(r) < \frac {\pi} 4 r^2$ follows  by integration, once we have verified for almost all  $t<r$  the inequality
\begin{equation} \label{b8}
v(t)\leq \frac 1 {\pi} w^2(t)  \;  .
 \end{equation}

 We claim that for all $t$ with finite $w(t)=\mathcal H^1(L_t)$, the set $L_t$ contains  a closed subset $L'_t$ still separating
 $x$ and $y$ in $Y_0$, and such that $L_t'$ is either a Jordan  curve   or homeomorphic to an open interval.  Indeed, consider the topological sphere $S^2=Z/\partial Z$ obtained from $Z$ by collapsing $\partial Z$ to a point, equipped with the quotient metric. Let  $K$ be the closure of the image of $I(L_t)$ in $S^2$. Then $K$ separates the (images of the) points $x$ and $y$ in $S^2$. Since $w(t)<\infty$ it follows that $K$ contains a Jordan curve
 $K_0$ which separates $x$ and $y$, cf. \cite[Corollary 7.5]{LW-intrinsic}.  The preimage $L' _t$ of $K_0$ in $Y_0$ is either a Jordan curve or an open interval, it is contained in $L_t$ and  separates $x$ and $y$ in $Y_0$.

Assume that $L'_t$ is a Jordan curve. By the choice of $x$, the Jordan domain of $L'_t$ contains $y$ and therefore the whole ball  $B_{Y_0}(y,t)$. Therefore,
 $v(t)  \leq \frac 1 {4\pi} \ell ^2 _Z(I(L'_t))  \leq \frac 1 {4\pi} w^2 (t)$, hence \eqref{b8}.

 Assume now that $L'_t$ is an open interval, which has finite length by assumption.  Then the closure of $I(L_t')$ consists
 of $I(L_t')$ and one or two points on $\partial Z$. By construction and assumption, these points are contained in
   $\overline{I(B_{Y_0}(y, r))}\cap \partial Z\subset c$. We consider the Jordan curve $T_t \subset Z$  given  by $I(L_t ')$ and    the part of $c$ between the endpoints of $I(L_t')$. Again, by the choice of $x$, the Jordan domain of $T_t$ contains $I(y)$ and therefore the image $I(B_{Y_0}(y,t))$.
On the other hand, the length of $T_t$ is at most $2w(t)$: since $c$ is a geodesic in $Z$, the curve $I(L'_t)$ has at least half of the length
of $T_t$.
 Thus, the isoperimetric inequality gives us $v(t) \leq \frac 1 {4\pi} (2w(t))^2 =\frac 1 {\pi} w^2(t)$, finishing the proof of \eqref{b8} and of the lemma.
\end{proof}

Under the assumptions of \lref{lem:simp} we  conclude:

\begin{prop} \label{prop:simplification}
The space $Y$ is  compact and homeomorphic to $\bar D$.
The restriction $I:\partial Y \to \partial Z$ is weakly monotone,  in particular,
the preimage $\eta =I^{-1} (c)$ is an arc.
The restriction $I:Y\setminus \eta \to Z\setminus c$ is a bijective local isometry.
The map $I:Y\to Z$ is an isometry if and only if $I$ maps $\eta$ to $c$ in an arclength preserving way.
\end{prop}

\begin{proof}
Assume that $Y$ is not compact. Then we find some $\epsilon >0$ and an infinite sequence of points $y_n \in Y_0$ with pairwise distance at least $2\epsilon$.
After choosing a subsequence we may assume  that $I(y_n)$ converges to a point $z \in Z$. If the point $z$ has a ${\rm CAT}(0 )$ neighborhood  in $Z$,
then for a small ball $B$ around $z$
 the metric of $B\cap (Z\setminus \partial Z)$ is a length metric, due to Subsection \ref{1to2}. From this we get $d_Y(y_n,y_m)=d_Z(I(y_n),I(y_m))$ for all $n,m$ large enough.  Therefore, the points $y_n$ cannot be $2\epsilon$-separated.

 Therefore,  we may assume that $z$ does not have a ${\rm CAT}(0 )$ neighborhood. Hence $z\in c([a,b]) \subset c([a_0,b_0])$.
    Thus,  we  find some small $r<\epsilon $ such that $B(z,3r) \cap \partial Z \subset c$.  This allows us to
 apply \lref{lem:areabound}  to $y_n$  and deduce, that  for all sufficiently large $n$ the area of the ball
 $B_{Y_0}(y_n,r)$ in $Y_0$ is at least   $\frac {\pi} 4 r^2$. But all these balls are disjoint. This contradicts the finiteness of the total area of $Y_0$ and finishes the proof that $Y$ is compact.
 From  \lref{lem:topology} we deduce that $Y$ is homeomorphic to $\bar D$.

Since $I(Y_0)$ is dense in $Z$ and $Y$ compact, we obtain   $I (Y)=Z$.  Since $I:Y_0\to Z\setminus \partial Z$ is a homeomorphism, we infer that $I^{-1} (\partial Z) =\partial Y$.
 An open subset $U$ of $\bar D$ is contractible if and only if $U\cap D$ is contractible.  Since $I:Y \setminus \partial Y\to Z \setminus \partial Z$ is a homeomorphism,
 we deduce that preimages of open contractible sets are contractible. Therefore, the preimage of any point $z\in Z$ is a cell-like set  (cf. \cite[p.97]{HNV04} or \cite[Section 7]{LW-intrinsic}). For  $z\in \partial Z$, the preimage is a cell-like subset of the circle $\partial Y$, hence either a point or an arc.  Therefore, the restriction $I:\partial Y \to \partial Z$ is weakly monotone.

In particular, the preimage $\eta$ of the arc $c\subset \partial Z$ is an arc in $\partial Y$.  For any point $z\in Z\setminus c$ there is a small ball $O$ around $z$
such that the metric on $O\setminus \partial Z$ is a length metric by Subsection \ref{1to2}. Thus,  $I:I^{-1} (O) \to O$ is an isometry.

Assume now  that  $I$  is arclength preserving on $\eta$. Then
$I$ is bijective and hence a homeomorphism.
  Moreover, for any geodesic $\gamma$ in $Z$, the map  $I:I^{-1} (\gamma ) \to \gamma$ preserves
the $\mathcal H^1$-measure. Thus, $I^{-1} (\gamma )$ has the same length as $\gamma$ and the map $I^{-1}:Z\to Y$ is $1$-Lipschitz.
Therefore, $I$ is an isometry.
 \end{proof}

In the sequel we will construct Jordan curves $T\subset Y$ whose intersection with $\eta$ is an arc $\eta_0$.  In this situation the image $I\circ T$ is a Jordan curve in $Z$ whose Jordan domain is the (locally isometric) image of the Jordan domain of $T$. The complementary part $k=T\setminus \eta_0$ is mapped by $I$ in an arclength preserving way.  Since $I$ is $1$-Lipschitz and the image of $\eta _0$ is a geodesic, we get
$$\ell _Z(I\circ T) = \ell _Y(k) +\ell _Z( I\circ \eta _0) \leq 2 \ell _Y(k) \; .$$
Moreover, the number $\ell _Y (T) -\ell _Z(I\circ T)= \ell _Y (\eta _0) -\ell _Z(I\circ \eta _0)$ measures the deviation of  $I$  from being an isometry.

\section{Rectifiable parts}  \label{sec:rect}
\subsection{Formulation}
We continue to work under the standing assumptions of \lref{lem:simp} and use the notations of  \pref{prop:simplification}.
The aim of this section is to prove:
\begin{prop} \label{prop:rect}
The map $I:Y\to Z$ preserves the length of any rectifiable subcurve of $\eta$.
\end{prop}

\subsection{Euclidean domains of almost isoperimetric equality}
We begin with a short Euclidean computation. For any sufficiently small $r>0$, let $T=T_r$ be a Jordan curve in $\R^2$
which consists of an arc of length $2\pi -2r$ on $S^1$ and a chord of length $ 2 \sin r$.
Then $\ell (T) < 2\pi$ and the Jordan domain $J$ of $T$ has area
$$\mathcal H^2(J)=\pi- r  +\frac {\sin (2r)} 2 > \pi - r^3 \; .$$
Therefore, we deduce:
$$\ell (T) -\sqrt {4\pi   \mathcal H^2(J)} < 2r^3\; .$$
We note that the curve $T$ is contained in a hinge of angle $\pi-  r $ enclosed between the chord and the tangent
to $S^1$ at one of the endpoints of the chord.  Rescaling  the curve $T$ suitably, we obtain:

\begin{lem} \label{lem:Eucl}
For any $\epsilon >0$ there exist some $L,\delta >0$ with the following property.
Let $\gamma :[0,\infty ) \to \R^2$ be one of two rays bounding a hinge $H$ of angle $\geq \pi- \delta$.
For any $s>0$ one can find a Jordan curve $T_s\subset H$ with Jordan domain $J_s$ such that the following holds true.
\begin{enumerate}
\item The curve $T_s$ contains  the initial part of $\gamma$ of length $s$.
\item $\ell (T_s) \leq L\cdot s$.
\item $\ell (T_s) -\sqrt {4\pi   \mathcal H^2(J_s)}  < \epsilon \cdot s$.
 \end{enumerate}

\end{lem}
\subsection{Curved domains of almost isoperimetric equality}
We   apply the approximation of hinges by flat cones, provided by  \lref{almostcone+}, and directly deduce from \lref{lem:Eucl}:

\begin{lem}  \label{almostflatgeod}
For any $\epsilon >0$ there exist some $L,\delta >0$ with the following property.
Let  $M$ be a metric space of non-positive curvature    homeomorphic to a surface with boundary. Let  $H$ be a hinge in $M$
of   angle  $ \geq \pi -\delta$ and  let $\gamma $ be one of its bounding geodesics.
 Then for all sufficiently small $s>0$
 there exists
a Jordan curve $T_s \subset H  \subset M$ with Jordan domain $J_s$ such that the conclusions (1)-(3) of \lref{lem:Eucl} hold true.
\end{lem}

 The bound $s_0>0$, such that the conclusion of \lref{almostflatgeod} holds true for all $0<s<s_0$,
depends on $\epsilon$, the space $M$, and the hinge $H$.

\subsection{Differentials} \label{subsecdiffer}
In order to approach general rectifiable curves we will use a Rademacher-type theorem.
Let  $\gamma :[p,q]\to X$ be a rectifiable curve parametrized by arclength in a ${\rm CAT}(0)$ space $X$.
 We say that $\gamma$ is \emph{differentiable}  at the point $t\in (p,q)$ if the in- and outgoing directions of $\gamma$ are "almost defined by almost opposite geodesic directions". More precisely,  we require the following conditions to hold true for the curves  $\gamma ^{\pm} (s):= \gamma (t\pm s)$.  The angle between $\gamma ^{\pm}$ is well-defined (cf. \cite[3.6.26]{BBI01}) and equal to $\pi$, and, moreover, there are geodesics $\eta _n ^{\pm}$ starting at $\gamma (t)$, such that the angles between $\eta ^{\pm} _n$ and $\gamma ^{\pm}$ are well-defined and converge to $0$, as $n$ converges to $\infty$.

 In other words,  the angle between $\eta _n^+$ and $\eta _n^-$ converges to $\pi$ and for any $\delta >0$  and all sufficiently large $n$, there exists some $s_0>0$ such that for all $0<s<s_0$ we have $d(\gamma (t\pm s), \eta ^{\pm} _n  (s)) < \delta s$.
 The metric differentiability theorem  implies (\cite[Theorem 1.6]{Diff}):

\begin{lem}  \label{differ}
Let $X$ be a  ${\rm CAT} (0 )$ space and let $\gamma :[p,q]\to X$ be a rectifiable curve parametrized by arclength.
 Then for almost all $t\in [p,q]$ the curve $\gamma$ is differentiable at $t$.
\end{lem}

Now we are able to deduce  that many small parts of any rectifiable curve can be complemented  to Jordan curves
almost violating the Euclidean isoperimetric inequality. In order to avoid minor difficulties we restrict ourselves
to  boundary curves, the only case we will  need.

\begin{prop} \label{equality2}
For any $\epsilon >0$ there exists some $L>0$ with the following property.
Let $M$ be a metric space of non-positive curvature   homeomorphic to a surface with boundary and let
 $\gamma  :[p,q] \to \partial M$ be a part of the boundary of $M$  parametrized by arclength.
 Then for a set $S$ of full measure in $[p,q]$ and any $t\in S$ there exists some $r_0=r_0(t)>0$ such that the  following holds true. For any $s<r_0$  there exists a Jordan curve $T_s$ in $M$ with Jordan domain $J_s$
 such that
 \begin{enumerate}
 \item The intersection of $T_s$ with $\gamma$ is an arc which contains $\gamma |_{[t,t+s]}$.
 \item $\ell (T_s)\leq L\cdot s$.
\item $\ell (T_s) -\sqrt {4\pi   \mathcal H^2(J_s)}  < \epsilon \cdot s$.
\end{enumerate}
\end{prop}

\begin{proof} Choose some $L=L(\frac \epsilon 4)$ provided by \lref{almostflatgeod}.
Let $S \subset (p,q)$ be the set of points in which $\gamma$ is differentiable.  For any $t\in S$, and any $\delta >0$ we
find geodesics $\gamma ^{\pm}$ starting in $\gamma  (t)$ at an angle not less than $ \pi -\delta$, such that
$ d(\gamma (t \pm s), \gamma ^{\pm }  (s))  < \delta \cdot s $  for all sufficiently small $s$.

Let $H$ denote a small hinge in $M$ enclosed by the geodesics $\gamma ^{\pm}$.
If $\delta$  and $s$ are small enough, we apply \lref{almostflatgeod} and  find a Jordan curve $T'_ s \subset H\subset M$ with Jordan domain $J'_s$ and the following properties.
The curve $T'_s$ has length at most $2Ls$ and contains the initial part of $\gamma^+$
of length $2s$. Moreover,    $$\ell (T'_s) -\sqrt {4\pi   \mathcal H^2(J'_s)}  < \frac \epsilon 2 \cdot s \; .$$
Now we connect the point $\gamma ^+ (2s)$ with   a nearest point  $\gamma (t+\hat s)$ on  $\gamma$  by a  geodesic $c_s$.
By the choice  of  $\gamma ^+$,  the length of $c_s$ is at most
$2\delta s$.  Moreover, $|\hat s-2s|=d(\gamma ^+ (\hat s),\gamma ^+ (2s)) \leq 2\delta s +\delta \hat s$ by the triangle inequality. Therefore, for $\delta <\frac 1 2$, we deduce
$\hat s -2s < 4\delta \hat s \leq 8\delta s$.

 Since $\gamma^+$ is a geodesic,  $c_s$ does not intersect $T' _s$ outside of $\gamma ^+$. Throwing away the common part of
$c_s$ and $\gamma ^+$ we may assume that $c_s$ intersects $T'_s$ only at the initial
point of $c_s$.

 Let now the  curve $T_s$ arise from $T'_s$ by replacing the initial $\gamma^+$-part of $T'_s$ by $c_s$ and the corresponding arc
 of $\gamma$ between $\gamma (t)$ and $\gamma(t+\hat{s})$.  By construction, the Jordan domain $J_s$ of $T_s$ contains the Jordan domain $J'_s$ of $T'_s$, hence $ \mathcal H^2(J'_s) \leq \mathcal H^2(J_s)$.  Moreover, the length of $\ell (T_s)$ is at most  $\ell (T' _s) + 2\delta s  +8\delta s$.

 Once $\delta$ and $s$ have been chosen small enough, we see that the curve $T_s$ satisfies all requirements of the lemma.
\end{proof}

\subsection{Length preservation}
Now we can easily provide:
\begin{proof} [Proof of \pref{prop:rect}]
Let  $\gamma: [p,q] \to \partial Y$  be  a rectifiable subcurve contained in $\eta$.  We may assume $\gamma$ to be parametrized by arclength.
 Since the map $I$ is $1$-Lipschitz we have
$\ell _Z (I\circ \gamma ) \leq \ell _Y(\gamma)$. If the inequality is strict  then we find
some $\epsilon >0$    and a set $Q$ of positive measure in $[p,q]$ such that the following holds true.  For any $r\in Q$, there is some $\delta =\delta (r)>0$ such that, for all $h<\delta$, one has  $$\ell _Z( I\circ \gamma |  _{[r,r+h]})  \leq h\cdot (1- 2\epsilon ).$$
  Applying  \pref{equality2} we find some $r\in Q$ and,
   for all sufficiently small $h>0$,  we find a Jordan curve $T_h$ as in \pref{equality2} containing $\gamma _{[r,r+h]}$. Let $J_h$ denote the Jordan domain of $T_h$.
  Then $I(J_h)$ is the Jordan domain of the Jordan curve $I(T_h)$.  Since on $Y_0$ the map $I$ is a local isometry, we have
  $\mathcal H^2 (J_h)=\mathcal H^2(I(J_h))$.     By assumption and the $1$-Lipschitz property of $I$, we see that
  $\ell _Z(I\circ T_h) \leq \ell _Y (T_h) -2\epsilon \cdot h$.

  We deduce  $\ell _Z(I \circ T_h) -\sqrt {4\pi   \mathcal H^2(I(J_h))}  < - \epsilon \cdot h$, which contradicts the isoperimetric assumption of \tref{thm4}, (3). This finishes the proof.
\end{proof}

\section{Final steps} \label{sec:final}
\subsection{Formulation}
We continue  to use the notations from \pref{prop:simplification}.
The rest of the section is devoted to the proof of
\begin{prop} \label{prop:finale}
The map $I:Y\to Z$ is an isometry.
\end{prop}

Assume the contrary. Due to  \pref{prop:simplification} and \pref{prop:rect}, the  arc $\eta$ is not rectifiable. We fix a parametrization of $\eta$  as a simple curve $\eta:[p,q]\to Y$.

  For the convenience of the reader, we first outline  the main steps of the proof. In \lref{curverect} we will deduce from the non-rectifiability of $\eta$ the existence of points on $\eta$ at which the "differential" of $I$ is arbitrary small. We will then fix such a point  $y$ and search for a contradiction to the isoperimetric inequality in a small neighborhood of this point. In \lref{Euclideanarea}, we show that the area of small balls around $y$ is almost Euclidean and obtain  in \cref{cor:longlength} a bound on the length of any curve surrounding such a ball.  In \lref{close} we connect $y$ by a geodesic $\gamma$ with a nearby point on $\eta$ and prove that
$\gamma$ and $\eta$ are sufficiently close to each other, more precisely, they include (in a rather weak sense)  an angle whose tangent is at most $\frac 1 2$.  In the final subsection, we take two geodesics connecting $y$ with nearby points on $\eta$, lying on different sides of $y$.  If the angle enclosed between these geodesics is at least $\pi$ then we obtain a contradiction in the same way  as at the end of  the proof of \pref{prop:rect}  above. If the angle is smaller than $\pi$ then  we
 consider two points  on these  geodesics  with small distance $r$ from $y$. We connect these points by a "circular arc" inside the hinge (using \lref{almostcone+})  and we further connect these points  to some points on $\eta$ using  \lref{close}. The arising curve is relatively short but nevertheless  surrounds a sufficiently large ball around $y$, leading  to a contradiction with \cref{cor:longlength}.

\subsection{Bounding diameter by endpoints on $\eta$}
We claim:
\begin{lem} \label{nodive}
For any $p\leq t<r\leq q$ the diameter of $\eta |_{[t,r]}$ is at most
$ 10 \cdot d(\eta (t), \eta (r))$.
\end{lem}

\begin{proof}
Assume the contrary and set $l= d(\eta (t), \eta (r))$.  Since $Y_0=Y\setminus \partial Y$ is a length space,
 we find a simple curve $k$  in $Y$ of length  less than $2l$
 connecting $\eta (t)$ and $\eta (r)$ and building a Jordan curve $T$ together with $\eta|_{[t,r]}$.  The  Jordan domain   $ J$ of $T$  has area at most $\frac 4 {\pi} l^2$, since the length of the image of $T$ in $Z$  is at most $2\cdot \ell _Y (k) = 4l$.

 Since $\bar J$ (which contains  $\eta |_{[t,r]}$) has diameter larger than  $10l$,
 we find a point $x \in \bar J$ with distance at least $4l$ from the curve $k$.  Thus, the ball $B_Y( x, 4l)$ is contained
 in $\bar J$.
 From
 \lref{lem:areabound} we see that $\mathcal H^2 (J)\geq \frac \pi 4 (2l)^2 = \pi l ^2 $.  This  contradicts
  $\mathcal H^2 (J) \leq  \frac 4 {\pi} l^2$ and finishes the proof.
\end{proof}

\subsection{A consequence of non-rectifiability}
Recall that $I\circ \eta$ is a weakly monotone parametrization of the  geodesic $c\subset Z$.  From now on, we will consider the curve $c$ with this weakly monotone parametrization $c=I\circ \eta:[p,q]\to Z$, despite the fact that in the rest of the paper all geodesics are parametrized by arclength.  We are going to find points at which the "differential" of $I$ is arbitrary small, by using \lref{nodive} and the fact that $\eta$ is non-rectifiable.

\begin{lem} \label{curverect}
For any $\lambda >0$ there exists some $t \in [p,q]$ and $\epsilon >0$ such that
for all $s\in [p,q]$ with $|s-t| <2\epsilon$ we have $$ d(\eta (t), \eta  (s))  \geq  \lambda \cdot d(c(t),c (s))   .$$
\end{lem}

\begin{proof}
We assume the contrary and  take
some $\lambda >0$ for which the claim is wrong.  We are going to prove that $\eta $ is rectifiable, in contradiction to our assumptions.

Consider an arbitrary  $\epsilon  >0$.
  For any $t\in [p,q]$, we find some $t^{+} \neq t$ with $|t^+ -t| < 2\epsilon $ and  $d(\eta (t), \eta (t^+)) < \lambda \cdot d(c(t), c(t^+))$. If $t$ is  one of the endpoints $p$ or $q$, we set $t^-=t$. If not then, by continuity of $\eta$ and $c$, we find  $t^-$ arbitrarily close to $t$ on the other side of $t$ from $t^+$ such that $d(\eta (t^+), \eta (t^-)) < \lambda   \cdot d(c(t^+),c(t^-))$.

Denote by $I_{t}$ the  closed interval between $t^-$ and $t^+$, which by our choice has length smaller than $2\epsilon$.  Changing  the order if needed we may assume $t ^- < t^+$ for any $t$.   We find a finite covering of $[p,q]$ by some of these   intervals  $I_{t_1},...,I_{t_k}$, such that the intersection number of the covering is at most $2$.
We reorder the intervals and have $t_i^- \leq  t _{i+1} ^-  \leq  t_i ^+$, for all $i$.    Thus  all  endpoints of all the intervals $I_{t_i}$ define  a $2\epsilon$-fine subdivision   $p=s_1<s_1\leq \dots \leq s_{2k }=q$  of $[p,q]$.
Each of the intervals $[s_i,s_{i+1}$] is contained in exactly one or two of the intervals  $[t_j^-,t_j^+]$.
Due to  \lref{nodive}, we have $d(\eta (s_{i}), \eta (s_{i+1})) \leq 10\cdot d(\eta  (t_j^-),\eta (t_j^+)) $ in this case.
Summing up we deduce

$$\sum  _{i=1} ^{2k}  \, d(\eta (s_i), \eta (s_{i+1})) \leq  20 \cdot \left( \sum_{j=1} ^k\, d(\eta (t_j^+),\eta (t_j^-)) \right)  $$
$$<  20 \lambda  \cdot \left( \sum_{j=1} ^k d(c (t_j ^-) ,c(t_j ^+ )) \right)\leq  40 \lambda  \cdot \ell (c) \, .$$
 Since $\epsilon$ was arbitrary, we see that $40 \lambda \ell (c)$ provides an upper bound for the length of $\eta$, in contradiction to the non-rectifiability of $\eta$.
\end{proof}

\subsection{Setting}
We choose some large $\lambda >0$, to be determined later.
  We find $t\in [p,q] $ and some $\epsilon >0$ provided  by  \lref{curverect}.
  Since $Z$ is non-positively curved in neighborhoods of the boundary points  $c(a_0) =I(\eta (p))$ and $c(b_0)=I(\eta (q))$ (cf. \lref{lem:simp}), the map $I$ is a local
isometry in neighborhoods of $\eta(p)$ and $\eta (q)$. Hence, $t\in (p,q)$.
  In order to simplify notations we may and will assume $t=0$ and  $[-\epsilon, \epsilon] \subset (p,q)$. Set $y=\eta (0)$ and note that  $I(y)$ is not an endpoint of the geodesic $c$ in $Z$.
   We choose some $r_0>0$ such that $B (y,2r_0) \cap \partial Y \subset \eta |_{[-\epsilon, \epsilon]}$.

We can now  show that balls around $y$ have almost Euclidean area.  We emphasize, that the point $y$ and the radius $r_0$ depend on the choice of the constant $\lambda$.

\begin{lem}  \label{Euclideanarea}
For any $\alpha_0>0$ the following holds true.
If $\lambda$ has been chosen large enough then for  any $r<r_0$ the area
of $O(y,r)=B (y,r) \cap Y_0$  can be estimated by
$$ \mathcal H ^2 (O(y,r) )\geq (\pi -\alpha _0) r^2 \;.$$
\end{lem}

\begin{proof}
Approximating  $y$  by points in $Y_0$   we  obtain from \lref{lem:areabound}   the inequality $\mathcal H ^2 (O (y,r)) \geq \frac \pi 4 r^2 $ for all $r<r_0$.
In order to improve the bound, we argue as in  the proof of \lref{lem:areabound}.
We consider the distance function $f:Y_0\to \R$ defined by $f(z)=d(y,z)$.  Then
 $O(y,t)$ is the sublevel set $f^{-1} ((0,t))$.  Denote  by $L_t \subset Y_0$ the  level set $f^{-1} (t)$ and by $w(t)$ its length
$\mathcal H^1 (L_t)$.   Set $v(t)= \mathcal H^2 (O(y,t))$.
By the co-area inequality,  $w\in L^1 ([0,r_0])$ and    $\int _0 ^t w(s) \, ds \leq v(t)$ for almost all $t \in (0,r_0)$.
As in the proof of \lref{lem:areabound}, for almost all $t\in (0,r_0)$, the  set  $L_t$ contains an arc $L'_t$ which  connects two points on $\eta$ and separates $O(y,t)$ from some fixed point $x\in Y_0$ at large distance from $y$.

Denote by $w_1(t)\leq w(t)$ the length $\mathcal H^1 (L'_t)$.  The statement of the lemma follows by integration, once we have verified  for almost all $t\in (0,r_0)$ the inequality
\begin{equation} \label{eq:lambda}
v(t) \leq \frac 1 {4\pi} w_1 ^2(t) \cdot (1+g_{\lambda})\; ,
\end{equation}
where the constant $g_{\lambda}$ goes to $0$ as $\lambda $ goes to $\infty$.

We already know $v(t)\geq \frac \pi 4 t^2$, and we have seen in the proof of \lref{lem:areabound} that $\frac 1 {\pi} w_1^2 (t) \geq v(t)$.
Therefore, $w_1(t) \geq \frac {\pi} 2 t$.

 The endpoints $e^{\pm}$ of $L'_t$ must be contained in $\eta |_{[-\epsilon, \epsilon]}$ and lie on different sides of $y$. Moreover, by definition, $d(e^{\pm},y) =t$.  As in \lref{lem:areabound}, we consider the Jordan curve  $T_t$  built by $L'_t$ and the part of $\eta$ between $e^{\pm}$.
  From the choice of $\lambda$ and $ \epsilon$ we deduce that the length of  $I\circ (T_t\cap \eta)$ is at most $\frac 2 {\lambda} t$.
 Therefore,
$$\ell _Z (I\circ T_t) \leq  w_1 (t)+ \frac 2 {\lambda} t \leq (1+ \frac 4 {\pi \lambda} ) \cdot w_1 (t) \; .$$
Now, as in the proof of \lref{lem:areabound}, the isoperimetric inequality in $Z$ provides \eqref{eq:lambda} with $(1+g_{\lambda})=(1+\frac 4 {\pi \lambda })^2$.
This finishes the proof.
\end{proof}

As a consequence we get:
\begin{cor} \label{cor:longlength}
If $\lambda$ is large enough then for any $r<\frac 1 3 r_0$  the following holds true.
Any curve $k$ in  $Y\setminus B _Y (y,r)$ which connects  two points  in $\eta ([-\epsilon, \epsilon ])$  on different sides of $y$ satisfies
the inequality $\ell _Y (k) > (\pi +3) r$.
\end{cor}

\begin{proof}
 Assume the contrary.  Fix a sufficiently  small  $\alpha _0>0$, choose $\lambda >0$ such that \lref{Euclideanarea} holds true.
 Consider an arbitary $r<\frac  1 3 {r_0}$ and a curve $k$  violating the conclusion of the corollary.
Then we find a simple subcurve of $k$ which still connects two points $e^{\pm}$ in  $\eta ([-\epsilon,\epsilon])$ on different sides of $y$ and does not intersect $\eta$ between $e^{\pm}$.  We replace $k$ by this subcurve and consider  the Jordan curve $T$ built by $k$ and the part of $\eta$ between $e^{\pm}$.   The Jordan domain of $T$  contains  $O(y,r) =B(y,r) \cap Y_0$.  Therefore,
\begin{equation} \label{eq:nn}
\frac 1 {4\pi}  \ell ^2  _Z(I\circ T) \geq (\pi-\alpha _0 ) r^2 \; ,
\end{equation}
 due to \lref{Euclideanarea} and the isoperimetric inequality in $Z$.

  We have $\ell _Y(k) \geq \frac 1 2 \ell _Z(I\circ T)$. If $k$ does not intersect the $3r$-ball around $y$ in $Y$, then we may replace the term  $r^2$ by $(3r)^2$ on the right hand side of \eqref{eq:nn}.  In this case we   arrive at  a contradiction, once $\alpha _0$ is small enough
  ($\alpha _0 = \frac 5 9 \pi$ is sufficient here).

  If $k$ intersects the $3r$-ball around $y$ in $Y$, then $e^{\pm}$
 have distance at most $10 r$ to $y$, since $k$ has length at most $(\pi+3)r$.  Therefore, $I\circ (T\cap \eta)$ has length at most $\frac {20} {\lambda } r$. Hence,
 $$\frac 1 {4\pi}  \left(\ell _Y (k) +\frac {20} {\lambda} r\right)^2 \geq (\pi-\alpha_0) r^2 \; .$$
  If $\lambda$ has been chosen sufficiently large and $\alpha _0$ sufficiently small, this contradicts    the assumption  $\ell _Y (k)  \leq (\pi +3) r$.
\end{proof}

 \subsection{The boundary is close to a geodesic}
Consider the geodesic $\gamma  :[0,t_0] \to Y$ starting at $y$ and ending at $\eta (\epsilon )$. Let $P: Y\to \gamma$ be the nearest point projection, which is well-defined and $1$-Lipschitz since $Y$ is ${\rm CAT}(0)$.
  For any point $x\in Y$, denote by $\beta _x$ the shortest geodesic from $x$ to $P(x)$.  Then $P(\beta _x)= P(x)$.
In particular, for $x,w\in Y$, the
geodesics $\beta _x$ and $\beta_w$ are either disjoint or are sent by
 $P$ to the same point, their common endpoint.  Since any geodesic $\beta _x$ encloses an angle of at least $\frac \pi 2$ with (the initial part of)
 $\gamma$ at $P(x)$, we infer $d(y,\beta _x) =d(y,P(x))$. In other words,
 any point on $\beta _x$ has at least the same distance from $y=\gamma (0)$ as $P(x)$.
For topological reasons we have:

\begin{lem} \label{lem:weakly}
The composition $P\circ \eta :[0,\epsilon] \to \gamma$ is a weakly monotone parametrization of $\gamma$.
\end{lem}

Denote by $Q$ the union of all geodesics $\beta _x$, where $x$ runs over all points on $\eta |_{[0,\epsilon]}$.
By definition, $Q$ contains $\eta |_{[0, \epsilon ]}$ and $\gamma$.
Denote by $Q_0$ the intersection $Q\cap Y_0$ and consider the $1$-Lipschitz continuous function  $f:Q_0\to [0,t_0]$ which sends $x\in Q_0$ to the $\gamma$-parameter of $P(x)$, thus
$f(x)=d(P(x),y)$.      By definition, $f^{-1} (t)$ is exactly the (intersection with $Y_0$ of the) union of all geodesics $\beta _w$ which start on $\eta  |_{[0,\epsilon]}$ and end in $\gamma (t)$.

  For $t\in [0,t_0]$  we let
  $h(t)\geq 0$  be the infimum of lengths of all geodesics $\beta _x$ which start at  some point $x \in \eta  |_{[0,\epsilon]}$ and end at $\gamma (t)=P(x)$.   By definition, $h(t)$ equals the minimum of the distance function
   to the geodesic $\gamma$ on the compact set $P^{-1} (\gamma (t)) \cap \eta |_{[0,\epsilon]}$. By continuity, $P(\eta |_{[0,\epsilon]}) =\gamma$, thus  $h$ is well-defined.
   By compactness, for any $t\in [0,t_0]$ we find a geodesic $\beta ^t =\beta _x$ of length $h(t)$ which starts
   on $\eta |_{[0,\epsilon]}$ and  ends at $\gamma (t)$.
   By minimality, the geodesic $\beta ^t$  intersects $\eta$ only at the starting point.  Again by compactness,
   the function $h(t)$ is lower semi-continuous.

  We set $g(t)=\mathcal H^1 (f^{-1} (t))$ for $t\in [0,t_0]$.
   By construction, we have $h(t)\leq g(t)$ for all $t$.
 By the co-area inequality, $g$ is integrable and  for any $0\leq t<t' \leq t_0$  we have
 \begin{equation} \label{eq:rechteck}
 \mathcal H^2 (f^{-1} ((t,t')) )\geq  \int _t ^{t'} g(s) \; ds \, .
 \end{equation}

With these notations and preparations at hand we can now show that $\gamma$ and $\eta$ do not diverge  too
fast from each other.

\begin{lem} \label{close}
For all $t \in [0,t_0]$ there is some  $t\leq t' \leq  2t$ with  $h(t')  \leq \frac {t'} 2$.
Thus, there exists a geodesic $\beta ^{t'}$ of length at most $\frac {t'} 2$ starting at $\gamma (t')$ orthogonally to $\gamma$ and ending
on $  \eta |_{[0,\epsilon]}$.
\end{lem}

\begin{proof}
Let $\mathcal T$ be the set of all $t\in [0,t_0]$ for which the claim is true.
By definition $h(t_0)=0$, therefore $[\frac {t_0 } 2, t_0] \subset \mathcal T$.
By the semi-continuity of $h$, $\mathcal T$ is closed.
Assume that $\mathcal T\neq [0,t_0]$ and let $t_3\in (0, \frac {t_0} 2]$ be the
 smallest number such $[t_3,t_0] \subset \mathcal T$.

  Consider $t_2:=2 t_3$. From  the minimality  of $t_3$
  we infer that  $h(t_2)\leq \frac 1 2 t_2$  and that for any $t\in [t_3,t_2)$ the inequality
  $h(t)> \frac 1 2 t$ holds  true.   Due to the semi-continuity of $h$  and since $h(0) = 0$,
   there exists a largest $t_1\in [ 0,t_3)$ with $h(t_1) \leq \frac 1 2 t_1$.
   Summarizing, we have
   $$h(t_1)\leq \frac 1 2 t_1; \; h(t_2) \leq \frac 1 2 t_2; \; t_2> 2 t_1 \; \text{and} \;  h(t) >\frac 1 2 t  \;  \text{ for}\;   t\in (t_1,t_2) \; .$$
    We are going to derive a contradiction to the isoperimetric inequality.  Consider the geodesic $\beta ^{t_1}$ and $\beta ^{t_2}$. By construction, these geodesics do not intersect $\eta$ outside their endpoints.
    Moreover, $\gamma |_{(t_1,t_2)}$ does not intersect $\eta$,  since otherwise $h$ were equal to $0$ at the intersection point.  Thus $\beta ^{t_1}, \beta ^{t_2}, \gamma |_{(t_1,t_2)}$ and the part of $\eta$ between the endpoints of $\beta ^{t_1}$ and $\beta ^{t_2}$ constitute a Jordan curve $T$.  Due to \lref{lem:weakly},  the preimage $f^{-1} ((t_1,t_2)) $ is contained in the Jordan domain $J$ of $T$.      Since
    $g(t)\geq h(t) > \frac 1 2  t$ for all $t\in(t_1, t_2)$, we deduce from \eqref{eq:rechteck} that
    \begin{equation}
    \mathcal H^2 (J) > \int _{t_1} ^{t_2} \frac 1 2 s \; ds =\frac 1 4  (t_2 ^2 -t_1 ^2) \, .
    \end{equation}

 We now estimate the length of $I\circ T$ in $Z$  as follows.
  By assumption, the lengths of  $\beta ^{t_1}, \beta ^{t_2}, \gamma |_{(t_1,t_2)}$ sum up to at most
  $\frac 1 2 (t_1 +t_2) + (t_2-t_1)$.  Moreover, the distance of the starting point of $ \beta ^{t_2}$
   on $\eta$ from $y$ is at most $t_2+ \frac 1 2 t_2$.  Thus, the $\eta$-part of $I\circ T$ is mapped
   to a part of the geodesic $c\subset Z$ which has length at most $\frac 1 \lambda  \cdot \frac 3 2 t_2$.

   We  set $q=\frac {t_1} {t_2} <\frac 1 2$. The  isoperimetric inequality in $Z$ gives us
   $ \mathcal H^2 (J)  \leq \frac 1 {4\pi}  \ell ^2 _Z(I \circ\ T)$.
   Inserting the above estimates we infer:
   \begin{equation} \label{eq:lastin}
   \frac 1 4  (1 -q^2)  \leq  \frac 1 {4\pi}  \left(\frac 3 2 -\frac 1 2 q + \frac 3 {2 \lambda  } \right)^2 \, .
   \end{equation}
   The left hand side is at least $\frac 3 {16}$ since $q < \frac 1 2$.  The right hand side is at most $\frac 9 {16 \pi} (1+\frac 1{ \lambda})^2$. Since $ \pi > 3$ we obtain a contradiction if $\lambda $ is large enough.
   This finishes the proof of \lref{close}.
 \end{proof}

\subsection{Final conclusions}
We  look at the other side of $y$ and connect $y$ with $\eta (-\epsilon)$ by a geodesic $\gamma _1$. We  apply the same
considerations to $\gamma _1$ which we applied to $\gamma$ above.  We deduce that for all sufficiently small $t$ there is some $t \leq t'\leq 2t$ and  a
geodesic $\alpha ^{t'}$ from $\gamma _1(t')$ to a point on $\eta |_{[-\epsilon,0]}$ such that  $\ell (\alpha ^{t'}) \leq \frac {t'} 2 \leq t$
and $d(\alpha ^{t'} ,y) =t'\geq t$.

The contradiction is now achieved in two steps.
\begin{lem} \label{large}
The angle between $\gamma$ and $\gamma _1$ must be at least $\pi$.
\end{lem}

 \begin{proof}
 Assume the contrary. We first claim  that for all sufficiently small $t$ there exists a curve $k_t$ between $\gamma (t)$ and $\gamma _1(t)$ such that
 any point on $k_t$ has distance at least $t$ from $y$ and such that $\ell (k_t) \leq (\pi +\frac 1 2) \cdot t$.

 Indeed, if the angle between $\gamma$ and $\gamma _1$ is  not $0$ we apply  \lref{almostcone+} to the hinge between $\gamma $ and $\gamma _1$.
 Thus, for  any fixed $\delta >0$ and all sufficiently small $t$,
 we find a curve $k'_t$ of length at most $(1+\delta) \pi   t$ which connects points $\gamma ((1+\delta ) \cdot t)$ and $\gamma _1  ((1+\delta) \cdot t)$ as the image of the corresponding circular arc in the flat hinge under the almost isometric map $E$ provided by  \lref{almostcone+}. Moreover, the distance of any point on $k'_t$ to $y$ is at most $t$. In order to obtain the required curve $k_t $, we just need to
 connect the endpoints of $k'_t$ with $\gamma (t)$ and $\gamma _1(t)$ along $\gamma$ and $\gamma _1$, respectively.  On the other hand, if the angle between $\gamma $ and $\gamma _1$ is $0$ (or just sufficiently small), we can obtain the required curve $k_t$ for all sufficiently small $t$ as follows:  connect  $\gamma(2t)$ with $\gamma_ 1(2t)$ by a geodesic and then connect  $\gamma (t)$ with $\gamma (2t)$ and $\gamma _1(2t)$ with $\gamma _1 (t)$   along $\gamma$ and $\gamma _1$, respectively.

 Now, we consider a sufficiently small $t$ such that the curve $\beta ^t$ has length at most $\frac t 2$. Such $t$ exists by  \lref{close}. Moreover,
 we apply \lref{close} to the curve $\gamma _1$ instead of $\gamma$ and find some $t\leq t'\leq 2t$ and  a
geodesic $\alpha ^{t'}$ with the properties  provided by \lref{close} and discussed prior to the present lemma.

Let the curve $k$ be the  concatenation of $\beta ^t, k_t, \gamma_1|_{[t,t']}$ and $\alpha ^{t'}$.   By construction, the curve $k$ lies completely outside the ball $B(y,t)$, it connects two points on $\eta$ which lie on different sides of $y$ and the length of $k$ is at most
$$\ell (k) \leq \frac t 2 +(\pi  + \frac 1 2)t + t + t\leq ( \pi +3) t \;.$$

This contradicts \cref{cor:longlength} and  finishes the proof.
 \end{proof}

The final lemma is proven similarly to the final step in the rectifiable case, \pref{prop:rect}:

\begin{lem} \label{small}
The angle  between $\gamma$ and $\gamma _1$ is strictly smaller than  $\pi$.
\end{lem}

\begin{proof}
We assume the contrary and apply \lref{almostflatgeod} to the hinge $H$ between $\gamma$ and $\gamma _1$.
Thus,  for all sufficiently small $s$, we find  a Jordan  curve $T_s$ in the hinge $H$ which contains the initial part of $\gamma$ of  length
$s$, and such that  for the Jordan domain $J_s$ of $T_s$  we have
$$\ell (T_s) -\sqrt {4\pi   \mathcal H^2(J_s)}  < \frac 1 3  s \;.$$

We now choose $s$ to be such that $\beta ^s$ has length at most $\frac s 2$ and replace $\gamma |_{[0,s]} \subset T_s$  by
the concatenation of $\beta ^s$ and the part of $\eta$ between the starting point of $\beta ^s$ and $y$.  The arising Jordan curve $T'_s$ contains
$J_s$ in its Jordan domain. The image Jordan curve $I\circ T'_s$ has length at most
$$\ell _Z  (I\circ T'_s) \leq \ell _Y(T_s) -s +\frac s 2 + \frac {3s} {2\lambda} = \ell _Y (T) -\frac 1 2 s  + \frac {3s} {2\lambda} \;.$$

If we have chosen $\lambda > 9$,  then the curve $I\circ T'_s$ does not satisfy the isoperimetric inequality in $Z$.
\end{proof}

The contradiction between \lref{small} and \lref{large} finishes the proof of \pref{prop:finale} and therefore the proof of \tref{thm4}.

\vspace{3cm}

\centerline{\bf{Appendix.}}

\vspace{0.2cm}

\section{Generalization to non-zero curvature bounds}
\vspace{0.2cm}

We sketch the proof of \tref{thmmain}, which generalizes \tref{thmfirst} to the case of non-zero curvature bounds. We refer to \cite{ballmann} and  \cite{AKP} for basics on ${\rm CAT }(\kappa )$ spaces  and recall that Reshetnyak's majorization theorem holds for all $\kappa$. Thus, any closed curve
 $\Gamma$ of length smaller than  $R_{\kappa}$ in any $CAT(\kappa )$ space is majorized by a convex subset in $M^2 _{\kappa}$. Now, the  proof
  of \lref{lem:onlyif} shows the "only if part" of \tref{thmmain}.

  Starting the proof of the "if part", assume that $X$ satisfies the conditions of \tref{thmmain}.  Due to  $\lim _{r\to 0} \frac {\delta _{\kappa} (r)} {r^2} =\frac 1 {4\pi}$, the arguments from Section \ref{secexclude} remain valid and prove that $X$ satisfies  property (ET).
Arguing as in Subsection \ref{subsec:reduct} we reduce the proof  of the "if part" to the following claim:  every intrinsic minimal disc  $Z$ in $X$ corresponding to a
solution of the Plateau problem $u\in \Lambda (\Gamma,X)$  is a  ${\rm CAT }(\kappa )$ space.  Here $\Gamma$ is any  Jordan curve in $X$ of length smaller than $R_{\kappa}$.

  The isoperimeric property of $X$ implies the same isoperimetric property  for all Jordan curves in $Z$. Thus, as in Section \ref{secisop},
  we deduce that the conformal factor $f$ of $u$ satisfies the integral  inequality (for all $z\in D$ and almost all $0<r<1-|z|$):
  \begin{equation}\label{eq:isop-fd}
   \int_{B(z,r)}f^2 \leq     \delta_{\kappa} \left(\int_{\partial B(z,r)}f\right) \;.
 \end{equation}

Arguing as in Sections \ref{secisop} and \ref{secsurf} and using \cite{Resh-isop} instead of \cite{BR} we obtain a metric of curvature $\leq \kappa$
on the disc $D$ which is defined on $D$ by the canonical semi-continuous representative of the conformal factor $f$.
  As in Section \ref{sec:mainred}
we reduce the proof to the following analogue of \tref{thm4}.
\begin{thm} \label{lastthm}
 Let $Z$ be a geodesic metric space homeomorphic to  $\bar D$. Assume that  for any Jordan curve
 $\Gamma$ in $Z$ the Jordan domain $J$ enclosed by $\Gamma$ satisfies $\mathcal H^2 (J) \leq \delta _{\kappa}  (\ell (\Gamma))$. Assume further that
 $Z\setminus \partial Z$ has curvature $\leq \kappa$. Then $Z$ is ${\rm CAT} (\kappa )$ and $Z\setminus \partial Z$ is a length space.
\end{thm}

To prove \tref{lastthm}  we closely  follow the second part of this paper. The approximation by flat cones \lref{almostcone+} is valid without changes
for all $\kappa \neq 0$. As  in Subsection \ref{1to2},  the ${\rm CAT} (\kappa )$ property of $Z$ implies that $Z\setminus \partial Z$ is a length space.  In order to prove that $Z$ is ${\rm CAT} (\kappa )$, we need the following additional lemma which can be used instead  of the theorem of Cartan-Hadamard.
\begin{lem} \label{verylastlem}
If $Z$ has curvature $\leq \kappa$ then $Z$ is ${\rm CAT} (\kappa )$.
 \end{lem}
Assuming that the lemma is wrong  we  obtain an isometric embedding into $Z$ of a circle $\Gamma$  of length
$2i<R_{\kappa}$, where $i$ is the injectivity radius of $Z$, cf. \cite[Section 6]{ballmann}.  Then one can either obtain a contradiction by
directly estimating the area of the Jordan domain $J$ of $\Gamma$ which contains a rather large metric ball, or apply  the fact that a round hemisphere is a minimal filling of a circle (cf. \cite{Ivfinsler})  to deduce that $\mathcal H^2(J) \geq \frac 1 {2\pi} \ell ^2(\Gamma )$ which contradicts the isoperimetric inequality.

\lref{verylastlem}  shows that the closed Jordan domain of any Jordan polygon in $Z\setminus \partial Z$ is ${\rm CAT}(\kappa )$ in its intrinsic metric.
Thus as in Section \ref{seccomplete}, we obtain that the completion $Y$ of the space $Y_0=Z\setminus \partial Z$, equipped with the induced length metric, must be a ${\rm CAT} (\kappa )$ space. From here the rest of the proof goes without changes, we only  need to restrict the attention   to sufficiently small distances, where $\delta _{\kappa} $ almost coincides with $\delta _0$.

\end{document}